\documentclass{amsart}
\usepackage{amsmath,amsfonts,amscd,amsthm}
\newtheorem{lemma}{Lemma}
\newtheorem{theorem}{Theorem}

\newcommand{\s}{\sigma}

\newcommand{\tx}{\tilde x}

\newcommand{\lp}{\Delta}
\newcommand{\la} {\langle}
\newcommand{\ra}{\rangle}

\begin{document}

\title{Curvature-Adapted Hypersurfaces of 2-type\\
in non-flat quaternionic Space Forms}

\author{Ivko Dimitri\'c}

\thanks{
\newline 2000 Mathematics Subject Classification: 53C40, \;53C42.
\newline
Key words and phrases: quaternionic space form, Laplacian,
curvature-adapted hypersurface, finite Chen-type, tube about a submanifold, principal curvatures of a hypersurface.}

\begin{abstract}
In a curvature-adapted hypersurface $M$ of a quaternionic-K\"ahler manifold $\overline M$  
the maximal quaternionic subbundle $\mathcal D$ of $TM$ and its orthogonal complement $\mathcal D^{\perp}$ in $TM$ are invariant subspaces of the shape operator at each point.
We classify curvature-adapted real hypersurfaces $M$ of non-flat quaternionic space forms $\mathbb HP^m$ and $\mathbb HH^m$ that are 
of Chen type 2 in an appropriately defined (pseudo) Euclidean space of quaternion-Hermitian matrices, where in the hyperbolic case we assume additionally that the hypersurace has constant principal curvatures. The 
position vector of a such submanifold in the ambient (pseudo) Euclidean space is decomposable into a sum of a constant vector and two nonconstant 
vector eigenfunctions of the Laplace operator of the submanifold  belonging to different eigenspaces. In the quaternionic projective space they include geodesic hyperspheres of arbitrary radius $r \in (0, \pi/2)$
except one, two series of tubes about canonically embedded quaternionic projective spaces of lower dimensions and two particular tubes about a
canonically embedded $\mathbb CP^m \subset \mathbb HP^m $.
 On the other hand, the list of 2-type 
curvature-adapted hypersurfaces with constant principal curvatures in $\mathbb HH^m$ is reduced to geodesic spheres and tubes of arbitrary radius about totally geodesic 
quaternionic hyperplane $\mathbb HH^{m-1}.$ Among these hypersurfaces we determine those that are mass-symmetric or minimal. We also show that the horosphere $H_3$ in $\mathbb HH^m $ is not of finite type but satisfies $\Delta^2\widetilde x =$ const.
\end{abstract}

\maketitle

\section{\bf Introduction}

\vskip 0.4truecm
\par In the theory of immersions of finite type developed by B. Y. Chen \cite{Che}, 
a submanifold $M^n$ of Euclidean or pseudo-Euclidean space isometrically immersed by $x: M^n \to E^N_{(K)} $ is said to be of finite type
in $E^N_{(K)}$ if the position vector $x$ can be decomposed into a finite sum of vector eigenfunctions of the Laplacian $\Delta _M$
on $M,$ viz.
\begin{equation}\label{eq1} x = x_0 + x_{t_1} + \dots + x_{t_k};  \qquad
\lp x_{t_i} = \lambda_{t_i} x_{t_i}, \;\; i= 1,...,k, \quad \lambda_{t_i} = \text{const} ,
\end{equation}
where $ x_0 = \text{const},\,  x_{t_i} \neq \text{const}, $ and the Laplacian acts on a vector-valued function 
componentwise. If $\lambda_{t_i} \in \Bbb R$ are all different, the immersion is said to be of Chen-type $k$  or simply of $k-$type. This 
definition is a generalization of the notion of minimal submanifolds of a  hypersphere or the ambient Euclidean space, whose
Chen-type is 1. For a compact submanifold, the constant part $x_0$ is the center of mass and if $x$ immerses $M^n$ into a central hyperquadric
of a Euclidean or pseudo-Euclidean space, the immersion is said to be mass-symmetric in that hyperquadric if $x_0$ coincides with the center of the said hyperquadric. Moreover, decomposition (\ref{eq1}) also makes sense for noncompact submanifolds, but $x_0$ may not be uniquely 
determined, namely in the case when one of the eigenvalues $\lambda_{t_i}$ above is zero (This is the case of submanifolds of null $k$-type).

\medskip
This notion can be extended to submanifolds
$x: M^n \to \overline M$ of a more general manifold $\overline M$ as long as there is a
``reasonably nice" (typically equivariant, with parallel second fundamental form) embedding\, $\Phi: \overline M \to E^N_{(K)}$ of the ambient 
manifold $\overline M$ into a suitable (pseudo) Euclidean space, in which case $M$ is said to be of Chen-type $k$ (via $\Phi$) if the 
composite immersion $\widetilde x:=\Phi \circ x$ is of Chen-type $k$.
\par
Symmetric spaces of rank 1 (spheres and projective and hyperbolic spaces),  allow some of the simplest
equivariant embeddings into a certain (pseudo) Euclidean space
$E^N_{(K)}$ of suitable Hermitian matrices over an appropriate field $\mathbb F$ by the projection
operators (see, for example \cite{Tai}, \cite{Sak}, \cite{Che} \cite{Ros1}, \cite{GR}, \cite{Dim2}). In the case of $\mathbb FP^m,$ this is the so-called first standard embedding, achieved, up to a rigid motion,  by the basis of eigenfunctions
corresponding to the first nonzero eigenvalue of the Laplacian. We use the notation $\mathbb HQ^m := \mathbb HQ^m(4c), \; c = \pm 1,$ 
to denote jointly either of the two simply-connected quaternionic model space forms: the quaternionic projective space $\mathbb
HP^m(4)$ or the quaternionic hyperbolic space $\mathbb HH^m(-4)$ of the respective quaternionic sectional curvatures $4$ and $-4.$  We shall assume $m \geq 2$.
Further, denote by  $\Phi : \mathbb HQ^m \to
E^N_{(K)}$ the embedding that associates to every quaternionic line in $\mathbb H^{m+1}$ (or time-like quaternionic line in 
the hyperbolic case) the operator (i.e. its matrix) of the
orthogonal projection onto it. The study of finite-type
submanifolds $x: M^n \to \mathbb HQ^m $ is then the study of the
spectral behavior of the associated immersion $\widetilde x = \Phi
\circ x$ of $M^n$ into $E^N_{(K)},$ i.e. of the possibility of
decomposing $\widetilde x$ into finitely many eigenfunctions of $\Delta_M.$ It is therefore interesting to investigate to what extent the
analytic information contained in the spectral resolution (\ref{eq1}) of the immersion into finitely many terms, determines the geometry
of the submanifold.
\par
A $k-$type immersion $x$ satisfies
a polynomial equation in the Laplacian, $P(\Delta )(x - x_0) = 0,$
where $P(t)$ is a  monic polynomial of the least degree ($k$) with
coefficients that  are the elementary symmetric functions of the
eigenvalues $\lambda_i, \; i = 1, \cdots , k $ \cite{Che}. Thus,
in the study of $k-$type submanifolds one necessarily deals with
the $k-$th iterated Laplacian $\Delta ^k x$ of the immersion,
which makes the investigation computationally very complex for
higher values of $k.$ Hence, the most promissing study involves 
submanifolds of low type: 1, 2, or 3. In particular, immersions $x: M^n \to E_{(K)}^N$ which are of 2-type  satisfy 
$\Delta^2 x - a \Delta x + b(x - x_0) = 0$ for some constants $a$ and $b$ and their study is related to the study of (extrinsically) biharmonic 
submanifolds satisfying the condition $\Delta^2 x = 0.$
\par
Regarding the situation in complex space forms, the study of 1-type
submanifolds of $\mathbb CP^m$ was begun in \cite{Ros1} and a
parallel investigation for hypersurfaces of $\mathbb CH^m$ was
carried out in \cite{GR}. A complete classification of 1-type
submanifolds of a non-flat complex space form $\mathbb CQ^m (4c)$
was achieved in our paper \cite{Dim1}. Submanifolds of Chen-type 2 have been
studied by several authors. Compact K\"ahler submanifolds of
$\mathbb CP^m$ were characterized by A. Ros \cite{Ros2} and
subsequently classified by Udagawa \cite{Uda2}, \cite{Uda3}, as compact
Einstein-K\"ahler parallel submanifolds, an example being the complex
quadric $Q^{m-1} \subset \mathbb CP^m$.
Minimal surfaces of 2-type were classified by Shen
\cite{She} and the classification of real hypersurfaces was undertaken
in \cite{MR} and  \cite{Uda4}, for hypersurfaces of
$\mathbb CP^m$ with constant mean curvature, and in \cite{Dim5}
for Hopf hypersurfaces of $\mathbb CQ^m $. Those 2-type Hopf (or CMC) hypersurfaces in $\mathbb CP^m(4)$ include all geodesic spheres except 
the one which is of 1-type, two series of tubes about canonically embedded (totally geodesic) 
$\mathbb CP^k(4) \subset \mathbb CP^m(4)$ for $k = 1, ..., m-2, $ and two particular tubes about the complex quadric $Q^{m-1}$. In a complex hyperbolic space the 
list of 2-type Hopf hypersurfaces includes only geodesic spheres and tubes about totally geodesic complex hyperbolic hyperplanes.
Submanifolds of non-flat real space forms which are of low type in a suitable (pseudo) Euclidean space of matrices via the immersion 
by projectors have been studied by several authors, see \cite{Dim4} and references there.

\vskip 0.3truecm
Among other rank-1 symmetric spaces, results on finite-type submanifolds of quaternionic space forms
and Cayley planes are rare and those are not studied as extensively in these ambient spaces because of a more complex expression for the curvature and higher dimension of the distribution $\mathcal D^{\perp}.$
\par
1-Type submanifolds of quaternionic projective spaces were investigated
in \cite{Dim2} where quaternion CR and anti-CR submanifolds of 1-type were classified. In particular, the only 1-type hypersurfaces 
of $\mathbb HP^m$ (up to rigid motions) are open portions of a geodesic  hypersphere of radius $r = \cot^{-1}\sqrt{3/(4m+1)}$. 
As a proper generalization of Hopf hypersurfaces of complex space forms to the quaternionic setting we have the notion of curvature-adapted
hypersurfaces of $\mathbb HQ^m, $ for which the normal Jacobi operator commutes with the shape operator, \cite{Dat}, \cite{Ber}, \cite{BV}. They are 
characterized by an equivalent condition that the maximal quaternionic subbundle $\mathcal D$ of the tangent space is an invariant subspace 
of the shape operator (and thus its orthogonal
complement $\mathcal D^\perp$ in $TM$ is also invariant) \cite{Ber}. 
In this paper we study curvature-adapted hypersurfaces of 
$\mathbb HQ^m$ of Chen-type 2, extending the results obtained for 2-type Hopf hypersurfaces in complex space forms to curvature-adapted ones in quaternionic space forms. Namely, we classify curvature adapted hypersurfaces of $\mathbb HP^m$ and $\mathbb HH^m$ (in the latter case assuming also constant principal curvatures) that are of 2-type in a suitably defined (pseudo) Eucldean space $E^N_{(K)}$ of Hermitian matrices. They include geodesic hyperspheres of arbitrary radius except one, tubes about $\mathbb HH^{m-1}$ in hyperbolic case, two series of tubes about quaternion projective spaces of lower dimension in $\mathbb HP^m$, and two particular tubes about canonically embedded $\mathbb CP^m \subset \mathbb HP^m$, whereas the other tubes about $\mathbb CP^m$ are shown to be of 3-type. Among these 2-type hypersurfaces we determine those that are mass-symmetric or minimal. We indicate how spectral decomposition can be obtained for a mass-symmetric 2-type hypersurface and give an explicit spectral resolution of $\widetilde x$ for geodesic spheres and tubes about $\mathbb CP^m(4)$ of radius $r = \frac 1 2 \cot^{-1}(1/\sqrt m)$ in $\mathbb HP^m(4)$. For a horosphere in $\mathbb HH^m$ we prove that it is not of finite type but satisfies the condition $\Delta^2 \widetilde x = \text{const} \neq 0$ and we prove that there are no real hypersurfaces of type 1 in $\mathbb HH^m$. These results are given in Theorems 1-4 and Lemmas 2 and 6.
\par
For the basic setup of the situation in quaternionic space forms and notation we refer to \cite{Ber}, \cite{Che}, \cite{Dim2}, \cite{Ish}, \cite{MP},\cite{CR}, \cite{AM1}, and \cite{LPS}.

\vskip 0.8truecm
\section{\bf Preliminaries}

\vskip 0.8truecm

Let $\mathbb H^m$ denote the $m-$dimensional quaternion number space $(m \geq 2)$, considered as a left vector space over (noncommutative) scalar  field 
of quaternions $\mathbb H$ and let $\mathbb HQ^m:= \mathbb HQ^m(4c)$ denote a $m-$dimensional non-flat model quaternionic space form, that is, either 
the quaternionic projective space $\mathbb HP^m(4)$ or the quaternionic hyperbolic space $\mathbb HH^m(-4)$ of constant quaternionic 
sectional curvature $4c = \pm 4.$ 
As is well known, $\overline M:= \mathbb HQ^m $ is a quaternion-K\"ahler manifold equipped with a quaternion-K\"ahler structure 
$\mathcal J,$ which is a rank-3 vector subbundle of $\text{End}\, (T\overline M)$ as oriented Riemannian bundle given an appropriate bundle metric and satisfying the following: For each $p \in \overline M$
there exists an open neighborhood $\overline G $ of $p$ and local basis of sections $J_1, J_2, J_3 \in \Gamma (\mathcal J)$ of $\mathcal J$
over $\overline G$ such that for every $q \in \{1, 2, 3 \}$ the following hold
\begin{equation}\label{eq2}
J_q^2 = - I, \quad \text{and} \quad \la J_q X, Y\ra = - \la X, J_q Y \ra ,
\end{equation}
i.e. $J_q$ is an almost Hermitian skew symmetric endomorphism of $T\overline M|_{\overline G}$ for $q = 1, 2, 3.$
\begin{equation}\label{eq3}
J_q J_{q+1} = J_{q+2} = - J_{q+1} J_q \quad (\text{indices}\mod 3),
\end{equation}
\begin{equation}\label{eq4}
\overline\nabla_X J \in \Gamma (\mathcal J), \quad\text{for all}\quad X \in \Gamma (T\overline M), \; J \in \Gamma (\mathcal J),
\end{equation}
i.e. $\mathcal J$ is a parallel subbundle of $\text{End}\, (T\overline M).$ Here and in what follows 
$\Gamma $ is used to denote the set of all (local) smooth sections of a bundle.

Any triple $\{J_1, J_2, J_3 \}$ of local sections of $\mathcal J$ satisfying $(\ref{eq2})-(\ref{eq4})$ is called a canonical local basis of
$\mathcal J.$ As a consequence of these conditions there exist 1-forms $\theta_1, \theta_2, \theta_3$ on $\overline G$ such that
\begin{equation}\label{eq5}
\overline\nabla_X J_q = \theta_{q+2}(X) J_{q+1} - \theta_{q+1}(X) J_{q+2} \quad (\text{indices} \mod 3 ),
\end{equation}
for all $X \in \Gamma (T\overline M), \; q \in \{1, 2, 3 \}.$ The Riemannian curvature tensor $\overline R$ of $\mathbb HQ^m(4c)$
has the form
\begin{equation}\label{eq6}
 \overline R(X, Y)Z = c\, [ \la Y, Z\ra X - \la X, Z\ra Y + \sum_{q=1}^3( \la J_qY , Z\ra J_qX -
\la J_qX , Z\ra J_qY - 2 \la J_qX , Y\ra J_qZ ) \, ], 
\end{equation}
for an arbitrary canonical local basis $\{J_1, J_2, J_3 \}$ of $\mathcal J.$
\par
By using a particular (pseudo) Riemannian submersion one can construct $\mathbb HQ^m$ and its embedding $\Phi$ into a suitable
(pseudo) Euclidean space of matrices by means of projections to quaternionic lines. We refer to \cite{Dim2} where this construction is 
carried out in detail for $\mathbb HP^m$ 
and we note that a similar construction works, {\it mutatis mutandis}, also for $\mathbb HH^m$ since it is the same kind of construction used
in the complex setting for $\mathbb CQ^m,$ \cite{Dim5}, \cite{GR}, \cite{Ros1}. For $\mathbb HH^m$ setting see also \cite{CR}, \cite{AM1}, \cite{AM2}, \cite{LPS}, \cite{OP}.
Consider first the standard Hermitian form $\Psi_c$ on $\mathbb H^{m+1}$ given by 
$\Psi_c (z, w) = c  z_0 \overline w_0 + \sum_{j=1}^m  z_j \overline w_j, \; z, w \in \mathbb H^{m+1},$
with the associated (pseudo) Riemannian metric $g_c = \text{Re}\,\Psi_c$ and the quadric hypersurface
$N^{4m+3}:= \{ z \in \mathbb H^{m+1}\, | \, \Psi_c (z, z) = c\}.$ When $c = 1, \; N^{4m+3}$ is the ordinary hypersphere $S^{4m+3}$ 
of $\mathbb H^{m+1} = \mathbb R^{4m+4}$ and when $c = -1, \; N^{4m+3}$ is an indefinite hyperbolic space $H_3^{4m + 3}$ of index 3 in 
$\mathbb H^{m+1}_1.$ The orbit space under the natural action of the group $Sp(1) = S^3$ of unit quaternions
on $N^{4m+3}$ defines $\mathbb HQ^m(4c),$ which becomes the base space of a (pseudo) Riemannian submersion with totally geodesic fibers. 
The standard embedding $\Phi$ into the set of $\Psi_c-$Hermitian matrices $\mathbf H^{(1)}(m+1)$ is achieved by identifying  a point, that is a 
quaternion line (or a time-like quaternion line in the hyperbolic case) with the projection operator onto it. Then one gets the 
following matrix representation of $\Phi$ at a point $p = [z],$ where $z = (z_j) \in N^{4m+3}\subset \mathbb H^{m+1}_{(1)}$is a row vector:
$$ \Phi ([z]) = \begin{pmatrix} |z_0|^2 & c \bar z_0 z_1 & \cdots & c \bar z_0 z_m \\ \bar z_1  z_0 & c |z_1|^2  & \cdots & c \bar z_1 z_m\\
\vdots & \vdots & \ddots & \vdots \\ \bar z_m  z_0 & c \bar z_m  z_1 & \cdots & c |z_m|^2
\end{pmatrix} .
$$
The image $\Phi (\mathbb HQ^m)$ of the space form considered is the set $\{P\}$ of Hermitian projectors ($P^2 = P$) with trace 1 and 
is contained in the hyperquadric of $\mathbf H^{(1)}(m+1)$ centered at $I/(m+1)$ and defined by the equation 
\begin{equation}\label{eq7}    \langle P - I/(m+1),\, P - I/(m+1)\rangle = \frac{cm}{2(m+1)}, 
\end{equation}
where $I$ denotes the $(m+1)\times (m+1)$ identity matrix and the metric on $\mathbf H^{(1)}(m+1)$ is given by $\langle S, T\rangle = \frac c 2 \, Re\, \text{tr} (S T).$
The second fundamental form $\s $ of this embedding is parallel, i.e. $\overline\nabla \sigma = 0.$
The following formulas for the shape operator of $\, \Phi \,\, $ in the direction of $ \sigma (X,Y) $ follow from \cite{Dim2} and 
are extensions of the corresponding well-known formulas of A. Ros in the complex case (see, for example \cite{Ros1}, \cite{Ros2}, \cite{GR}, \cite{Dim2}).
\begin{eqnarray}\label{eq8} \la \s (X,Y), \s (V, W) \ra &=& c\, [2 \la X, Y\ra \la V, W \ra +
 \la X, V\ra \la Y, W \ra + \la X, W\ra \la Y, V \ra \nonumber\\ 
&&\qquad  + \sum_{q=1}^3 ( \la J_qX, V\ra \la J_qY, W \ra + \la J_qX, W \ra \la J_qY, V \ra) ], 
\end{eqnarray}

\begin{equation}\label{eq9} \overline A_{\sigma (X,Y)}V = c\, [2 \langle X, Y\rangle V + \langle X,
V\rangle Y + \langle Y, V\rangle X + \sum_{q=1}^3 (\langle J_qX, V\rangle J_qY + \langle
J_qY, V\rangle J_qX )]. 
\end{equation}

One also verifies (see \cite{Dim2})
\begin{equation}\label{eq10}  \qquad \langle \sigma (X, Y), \widetilde x \rangle = - \langle X, Y
\rangle , \qquad \langle \sigma (X, Y), \, I \rangle = \, 0, \qquad \sigma (J_qX, J_qY) = \sigma (X, Y), \; q = 1, 2, 3. 
\end{equation}

For additional properties of the embedding $\Phi$ see \cite{Tai}, \cite{Sak}, \cite{Che}, \cite{Ros1}, \cite{MP}, \cite{Dim1}.
\par

Let $x: M^n \to \mathbb HQ^m$ be an isometric immersion of a Riemannian $n-$manifold as a real hypersurface of a quaternionic 
space form (thus $n = 4m-1$). Then we have the associated composite immersion $\widetilde x = \Phi \circ x,$ which realizes $M$ as a submanifold 
of the (pseudo) Euclidean space $E^N_{(K)}:= \mathbf H^{(1)}(m+1) $ of dimension $N = (m+1)(2m+1) $ and appropriate index $K$, equipped with the 
usual trace metric $\la A, B \ra = \frac c 2 Re\, \text{tr}\, (AB) $. In this notation the subscripts and superscripts in parenthesis 
are present only in relation to $\mathbb HH^m,$ so that the superscript 1 in $\mathbf H^{(1)}(m+1)$ is optional and appears only in 
the hyperbolic case, since the construction of the embedding is based on the form $\Psi$ in $\mathbb H_1^{m+1}$ of index 1.
\par
Let $\xi $ be a local unit vector field normal to $M$ in $\mathbb HQ^m ,$  $A$ the shape operator of the
immersion $x,$ and let $\kappa = (1/n)\,\text{tr}\, A $ be the mean curvature of $M$ in $\mathbb HQ^m,$ so that the mean
curvature vector $H$ of the immersion equals $H = \kappa \xi .$
Further, let $\overline\nabla , \; \overline A, \overline D,$ denote respectively the Levi-Civita
connection, the shape operator, and the metric connection in the normal
bundle, related to $\mathbb HQ^m$ and the embedding $\Phi$. Let the same letters
without bar denote the respective objects for a submanifold $M$ and the immersion $x,$
whereas we use the same symbols with tilde to denote the corresponding objects related to the
composite immersion $\tilde x : = \Phi \circ x$ of $M$ into the (pseudo) Euclidean space
$\mathbf H^{(1)}(m+1).$ As usual, we use $\s$ for the second fundamental form of
$\mathbb HQ^m$ in $E^N_{(K)}$ via $\Phi$ and $h$ for the second fundamental form of
a submanifold $M$ in $\mathbb HQ^m$. An orthonormal basis of the tangent space $T_pM$ at a general
point will be denoted by $\{e_i \}, i = 1, 2, {\ldots} , n,$ and indices $i, j$ will generally range from $1$ to $n$, whereas $q$ and $k$ range from 1 to 3. 
\par
We give first some important formulas which will be repeatedly used throughout this paper.
For a general submanifold $M, $ of a Riemanian manifold $\overline M, \,$ local tangent fields $X, Y \in \Gamma (TM)$ and a
local normal field $\xi \in \Gamma (T^\perp M),$ the
formulas of Gauss and Weingarten are
\begin{equation}\label{eq11} \overline \nabla_XY = \nabla_XY +
h (X, Y) ; \qquad\qquad \overline \nabla_X \xi = -A_\xi X + D_X\xi.  
\end{equation}

In particular, for a hypersurface of a quaternionic space form $\mathbb HQ^m $
with (locally defined) unit normal vector field $\xi $ and the corresponding shape operator $A = A_\xi,$ they
become
\begin{equation}\label{eq12} \overline \nabla_XY =\nabla_XY + \langle AX, Y \rangle \xi
;\qquad \qquad \overline \nabla_X\xi = - AX. 
\end{equation} 

Let $\{J_q\}, \, 1 \leq q \leq 3,$ be a triple of  almost complex structures of $\mathbb HQ^m$ which form a canonical basis of $\mathcal J$ 
and $U_q \in \Gamma (TM)$ be the tangent vector fields defined by $ U_q:= -J_q\xi , \, q = 1, 2, 3$. Define distributions $\mathcal D^\perp = Span_{\mathbb R} \{U_1, U_2, U_3  \}$ and $\mathcal D$ to be the orthogonal complement of $\mathcal D^\perp$ in $TM$ so that $\mathcal D$ is the maximal subbundle of $T\overline M$ which is left invariant by the quaternionic structure $\mathcal J|_{TM}$. 
Further, define endomorphisms
$S_q$ of the tangent space and a normal bundle valued 1-forms $F_q$ by
$$ S_qX = (J_qX)_T , \quad F_qX = (J_qX)_N  = \langle X, U_q \rangle \xi ,$$
i.e for $X \in \Gamma (TM), \;\; J_qX = S_qX + F_qX $ is the
decomposition of $J_qX$ into tangential and normal to submanifold parts. Further, we obtain the following from (\ref{eq2}) - (\ref{eq5}) and (\ref{eq11}):
\begin{equation}\label{eq13}
S_qU_q = 0, \quad  S_q U_{q+1} = U_{q+2}, \quad S_q U_{q+2} = - U_{q+1}
\end{equation}

\begin{equation}\label{eq14}
S_qX = J_qX - \langle X, U_q \rangle \xi , \qquad S_q^2X = -X + \langle X, U_q \rangle U_q,
\end{equation}

\begin{equation}\label{eq15}
\nabla_XU_q = S_qAX + \theta_{q+2}(X)\, U_{q+1} - \theta_{q+1}(X)\, U_{q+2}, 
\end{equation}

Note that both $\mathcal D$ and $\mathcal D^\perp$ are $S_q$-invariant, for each $q = 1, 2, 3$. The equations of Codazzi for a hypersurface of $\mathbb HQ^m(4c)$ is  given by
\begin{equation}\label{eq16}
(\nabla _XA)Y - (\nabla_YA)X = c\,\sum_{q=1}^3 [\langle X, U_q \rangle S_qY -
\langle Y, U_q \rangle S_qX - 2\langle S_qX, Y \rangle U_q]. 
\end{equation}

For any hypersurface $M$ of a Riemannian manifold $\overline M$ with curvature tensor $\overline R$ one defines the normal
Jacobi operator $\mathcal K \in \text{End}\, (T_pM) $ at a point $p \in M$ by $\mathcal K(X):= \overline R (X, \xi ) \xi $ where $\xi $ is
a unit normal to $M$ at $p$ (determined up to a sign) and $X$ an arbitrary tangent vector at $p.$ A hypersurface $M$ is said to be 
{\it curvature-adapted} to $\overline M$ if the normal Jacobi operator $\mathcal K$ of $M$ commutes with the shape operator $A$ at every point, i.e. the two operators are simultaneously diagonalizable. Equivalently, 
$\mathcal K\circ A = A\circ \mathcal K$ translates into $\overline R (AX, \xi ) \xi = A (\overline R (X, \xi ) \xi ),$ 
for every $X \in \Gamma (TM),$ and for a hypersurface of a quaternionic space form 
$\mathbb HQ^m (4c)$ of quaternionic sectional curvature $4c$ whose curvature tensor is given by (\ref{eq6}) this 
further leads to
\begin{equation}\label{eq17}
\sum_{q = 1}^3 \langle AX, U_q\rangle U_q = \sum_{q = 1}^3 \langle X, U_q\rangle AU_q, \quad \text{for every}\;\; 
X \in \Gamma (TM),
\end{equation} 
i.e. $(AX)_{\mathcal D^{\perp}} = A(X_{\mathcal D^{\perp}}),$ which in turn is equivalent to the condition that the distributions 
$\mathcal D^{\perp}$ and $\mathcal D$ in $TM$ are invariant under the action of the
shape operator $A,$ meaning $A \mathcal D^{\perp} \subset \mathcal D^{\perp} $ and $A \mathcal D \subset \mathcal D.$ The notion of curvature-adapted hypersurface was introduced in \cite{Dat} and effectively used in the study of hypersurfaces in \cite{BV}, \cite{Ber}, more recently in \cite{AM1}, \cite{AM2}, \cite{KN} and other works.
For such a hypersurface
$AU_q = \sum_{k=1}^3 a_q^k U_k $ for some real functions $a_q^k, $ symmetric in $k$ and $q,$ and then using
(\ref{eq13}) we get $\sum_q S_q A U_q = 0 $.

\vskip 0.3truecm
The gradient of a smooth function $f$ is a vector field
$\nabla f := \sum_i (e_i f)e_i$ and the Laplacian acting on smooth functions is defined as 
$\lp f := \sum_i [(\nabla_{e_i}e_i) f - e_i e_i f],$ where $\{ e_i\}, 1 \leq i \leq n$, is an orthonormal basis of the tangent space of $M.$
The Laplace operator can be extended to act on a vector field $V$ in $E^N_{(K)} = \mathbf H^{(1)}(m+1)$ along $\widetilde x
(M)$ by
\begin{equation}\label{eq18}
\Delta V = \sum_i [\widetilde\nabla_{\nabla_{e_i}e_i} V - \widetilde\nabla_{e_i}
\widetilde \nabla_{e_i} V].
\end{equation}
The product formula for the Laplacian, which will be used in the ensuing
computations, reads
\begin{equation}\label{eq19}
\Delta (f\, g) = (\Delta f)\, g + f(\Delta g) - 2 \sum_i (e_if)( e_ig) ,
\end{equation}
for smooth functions $f, \,g \in C^\infty (M)\, $ and it can then be extended to hold for the scalar
product of vector valued functions, hence also for product of matrices, in a natural way.  We shall use
the notation $f_k:= \text{tr}\, A^k, $ and in particular $f:= f_1 = \text{tr}\, A.$
For an endomorphism $B$ of the tangent space of $M$ we define
$\text{tr} (\nabla B): = \sum_{i=1}^n (\nabla_{e_i}B)e_i .$ We will use $V_\mu$ to denote vector space of principal vectors (eigenvectors of $A$) corresponding to principal curvature $\mu$ and $\frak s(\mathcal D), \frak s(\mathcal D^\perp)$ to denote the spectrum (the set of eigenvalues) of $A$ when restricted to $\mathcal D$ and $\mathcal D^\perp$, respectively, at a point considered.
\par
We shall assume all manifolds to be smooth and connected, and all immersions smooth.
\vskip 1truecm

\section{\bf The Second Iterated Laplacians of a Real Hypersurface of $\mathbb HQ^m$}

\vskip 0.6truecm
From a well-known formula of Beltrami we have 
\begin{equation}\label{eq20} 
\Delta \widetilde x = - n \widetilde H = - f\,\xi - \sum_{i=1}^n \sigma (e_i, e_i ),  
\end{equation}
where here, and in the following, we understand the Laplacian $\Delta $ of $M$ to be applied to
vector fields along $M$ (viewed as $E^N_{(K)}-$valued functions, i.e.  matrices) componentwise.

The product formula (\ref{eq19}) gives

\begin{equation}\label{eq21}
\Delta^2\widetilde x := \Delta (\Delta\widetilde x ) = - (\Delta f )\xi - f
(\Delta \xi ) + 2 \sigma (\nabla f , \xi ) - 2 A(\nabla f )- \sum_i
\Delta\, (\sigma (e_i, e_i )).  
\end{equation}
Moreover, we compute
\begin{eqnarray*}
\Delta\xi &=& \sum_i [\widetilde \nabla _{\nabla_{e_i}e_i}\xi -
\widetilde \nabla _{e_i}\widetilde \nabla_{e_i} \xi ]\\ 
&=& \sum _i [-A(\nabla_{e_i}e_i ) + \sigma (\nabla_{e_i}e_i, \xi ) + \overline {\nabla}_{e_i}(Ae_i) 
+ \sigma (e_i, Ae_i) + \overline A_{\sigma (e_i, \xi )} e_i - \overline{D}_{e_i} (\sigma (e_i, \xi ))]. 
\end{eqnarray*} 
Using (\ref{eq9}), the parallelism of $\sigma ,$  and the fact that $\, \text{tr}\, (\nabla A) = \sum_{i = 1}^n (\nabla_{e_i}A) e_i =  \nabla (\text{tr}\, A) = \nabla f $
(which follows from the Codazzi equation), we obtain
\begin{equation}\label{eq22}
\Delta\xi = \nabla f + [f_2 + c(n - 3)]\, \xi - f\, \sigma (\xi, \xi ) + 2 \sum_i \sigma (e_i, Ae_i).  
\end{equation}

\par Further computations yield
\begin{equation}\label{eq23}
\sum_i \widetilde \nabla_X \Big(\sigma (e_i, e_i)\Big) = - 2 c (n + 4) X + 2 c \sum_{q = 1}^3 \langle X, U_q \rangle U_q + 2 \sigma (AX, \xi), 
\end{equation}
so that starting from (\ref{eq18}) and using (\ref{eq11}) - (\ref{eq15}) and (\ref{eq23}) we get

\begin{eqnarray}\label{eq24} \sum_i \Delta ( \sigma (e_i, e_i)) &=& 2c \, [ (n+5)f - 2 \sum_{q=1}^3 \langle AU_q, U_q \rangle ]\,\xi 
-\, 4c\, \sum_{q=1}^3 S_qAU_q - 2(3c + f_2 )\, \sigma (\xi, \xi )\nonumber\\
&& \quad\;  - 2 \sigma (\xi, \nabla f) +  2c\, (n+4) \sum _i \sigma (e_i, e_i ) + 2 \sum_i \sigma (Ae_i, Ae_i).  
\end{eqnarray} 

Combining formulas (\ref{eq21}) - (\ref{eq24}) we finally obtain 

\begin{eqnarray}\label{eq25} \Delta^2 \tilde x &=& - [\Delta f + ff_ 2 + c( 3n + 7)f - 4c \sum_{q=1}^3 \langle AU_q, U_q \rangle ]\, \xi 
- f\, \nabla f - 2 A (\nabla f) \nonumber \\
&&\quad + \, 4c\, \sum_{q=1}^3 S_qAU_q  + ( 6c + 2 f_2 + f^2 )\, \sigma (\xi, \xi ) + 4\, \sigma (\nabla f , \xi )\\ 
&&\quad - \, 2 c(n+4) \sum_i \sigma (e_i, e_i ) - 2 f\,\sum_i \sigma (e_i, Ae_i ) -  2 \sum_i\sigma (Ae_i, Ae_i ) , \nonumber
\end{eqnarray} 
which holds for any real hypersurface of $\mathbb HQ^m(4c)$.
Compare this with a similar formula in \cite{Dim4}, \cite{Dim5} and formula (2.9) of \cite{GR}.
\vskip 0.6truecm

\section{\bf Curvature-Adapted Hypersurfaces of 2-Type With Constant Principal Curvatures }
\vskip 0.6truecm

 In this section we work with 2-type curvature adapted hypersurfaces of $\mathbb HQ^m$ that have constant principal curvatures.
As a matter of fact, it was proved in \cite{Ber}  that a curvature-adapted hypersurface of $\mathbb HP^m$ necessarily has constant principal curvatures
without any additional assumptions, but that is not yet known for hypersurfaces of $\mathbb HH^m.$
Assume that a general hypersurface $M^n \subset \mathbb HQ^m , \; n = 4m-1,$ is of Chen-type 2 via the embedding $\Phi$, i.e. 
$\widetilde x = \widetilde x_0 + \widetilde x_u + \widetilde x_v$ according to (\ref{eq1}). Then by taking succesive Laplacians and eliminating 
$\widetilde x_u , \widetilde x_v$ we get
\begin{equation}\label{eq26}
\Delta^2 \widetilde x - a \Delta \widetilde x + b \widetilde x = b \widetilde x_0 , \quad \text{with}\quad 
a := \lambda_u + \lambda_v, \;\; b := \lambda_u \lambda_v
\end{equation}
Conversely, (\ref{eq26}) implies that $M^n$ is of type $\leq 2 $ when $M$ is compact or when the trinomial $P(t) = t^2 - a t + b$ has two 
distinct real roots \cite{CP}. Denote the vector field along $\widetilde x (M)$ represented by the left-hand side of (\ref{eq26}) by $L.$ 
Then, differentiating $L$ with respect to an arbitrary tangent vector field $X$ and taking the inner product with $\widetilde x,$ 
using (\ref{eq10}), (\ref{eq20}) and (\ref{eq25}), gives
\begin{eqnarray*}
 0 &=& \langle \widetilde\nabla_X L, \widetilde x \rangle = X \langle L, \widetilde x \rangle - \langle L, X \rangle \\
 &=& X (f^2 - 6c + 2cn (n+4) - an + bc/2) + \langle f \nabla f + 2 A (\nabla f), X \rangle - 4c \sum_{q = 1}^3 \langle S_q A U_q, X \rangle ,
\end{eqnarray*}
from where we obtain
\begin{equation}\label{eq27}
2 A (\nabla f) + 3 f \nabla f = 4c \sum_{q=1}^3  S_q A U_q .
\end{equation}
 Consequently, (\ref{eq27}) yields 
\begin{equation}\label{eq28}
A (\nabla f ) = - \frac{3f}{2} \nabla f ,
\end{equation}
for a 2-type curvature-adapted hypersurface. Thus, on an open (possibly empty) set $\{ \nabla f \neq 0 \}$ the gradient of $f$ is a principal direction. Other conditions can be obtained by considering various components of differentiated 2-type equation (\ref{eq26}), but the expressions are rather complicated compared to $\mathbb CQ^m (4c)$ setting. Although those kinds of computations were instrumental in proving that a 2-type Hopf hypersurface of $\mathbb CQ^m$ has constant mean curvature and, moreover, constant principal curvatures, those conclusions do not follow in a similar way in quaternionic setting. Hence, for the rest of this paper, we consider now curvature adapted hypersurfaces of type 2 that have constant principal curvatures. This is not any restriction in $\mathbb HP^m$ case, since a curvature-adapted hypersurface of $\mathbb HP^m$ has (locally) constant principal curvatures and full classification of hypersurfaces with constant principal curvatures in both $\mathbb HP^m$ and $\mathbb HH^m$ is known \cite{Ber}.
The information on these model curvature-adapted hypersurfaces of $\mathbb HQ^m(4c)$  with constant principal curvaturess in terms of principal curvatures $\mu, \nu, \, \alpha_i,$ \, and their multiplicities $m(\mu), m (\nu), \, m(\alpha_i), \; i = 1, 2 $,  is presented in  Table 1 below taken from \cite{Ber}. The left portion of the table corresponds to the projective case and the right portion to the hyperbolic one. Here
$\mu$ and $\nu $ (resp. $\alpha_1$ and $\alpha_2$) are principal curvatures of $A|_{\mathcal D}$ (resp. $A|_{\mathcal D^{\perp}} $).
\par
These hypersurfaces are usually categorized as hypersurfaces of class (type) $A$ and class (type) $B$. Class-$A$ hypersurfaces are $P_1^k(r)$ in $\mathbb HP^m(4)$ which form a family of tubes of some radius $r \in (0, \pi/2)$ about a canonically embedded $\mathbb HP^k(4)$, and $H_1^k(r)$ in $\mathbb HH^m(- 4)$, which form a family of tubes of some radius $r > 0$ around a canonically embedded $\mathbb HH^k(- 4)$ for some $k \in \{0, 1, \dots , m - 1  \}$. A particular subclass $A_1$ is obtained when $k = 0$, producing a family of geodesic hyperspheres $P_1^0(r)$ and $H_1^0(r)$, or when $k = m - 1$ in the hyperbolic case, producing a family of equidistant hypersurfaces to canonically imbedded (totally geodesic) hyperplane $\mathbb HH^{m - 1}(- 4)$. Class-$A_2$ hypersurfaces are all the other tubes $P_1^k(r), \, H_1^k(r), \; k \neq 0,\, m - 1. $ The horosphere $H_3$ in $\mathbb HH^m(- 4)$ is of class $A_0$. Class-$B$ hypersurfaces are $P_2(r)$ in $\mathbb HP^m(4)$ which are tubes of some radius $r \in (0, \pi/4)$ about a canonically embedded complex projective space $\mathbb CP^m(4) $ of half the dimension and $H_2(r)$ in $\mathbb HH^m(- 4)$ which are tubes of some radius $r > 0$ about a canonically embedded complex hyperbolic space $\mathbb CH^m(- 4)$.
\vskip 0.3truecm
In a neighborhood of every point of some open dense set $\mathcal W$ of curvature-adapted hypersurface $M^n$ according to Lemma 3.6 of 
\cite{Ber} we can choose a local canonical 
basis $\{J_1, J_2, J_3 \}$ of quaternionic bundle $\mathcal J $ over some open neighborhood $\overline{G}$ of $\mathbb H Q^m$  and a local unit normal field of $M$ defined on $G: = \overline{G} \cap M \subset \mathcal W$ so that $U_q: = - J_q \xi ,  \, q = 1, 2, 3 $ are principal on $G$. Let $\alpha_q$ be 
the corresponding (locally constant) principal curvatures, i.e. $A (U_q) = \alpha_q U_q$ and let $\frak s (\mathcal D)$  and $\frak s (\mathcal D^\perp)$ denote the spectrum (collection of eigenvalues) of $A|_{\mathcal D}$ and $A|_{\mathcal D^\perp}$, respectively. 
Formula (\ref{eq25}) becomes
\begin{eqnarray}\label{eq29} \Delta^2 \widetilde x &=&  \Big\{ 4c \sum_{q=1}^3 \alpha_q - f\, [ f_ 2 + c( 3n + 7)] \Big\} \, \xi + 
( 6c + 2 f_2 + f^2 )\, \sigma (\xi, \xi ) \nonumber\\ 
&&\quad\;\; - 2 c(n+4) \sum_i \sigma (e_i, e_i ) - 2 f\,\sum_i \sigma (e_i, Ae_i ) -  2 \sum_i\sigma (Ae_i, Ae_i ) .
\end{eqnarray} 

Let $X \in \Gamma (TM).$ Using formulas (\ref{eq9}), (\ref{eq11}), (\ref{eq12}), (\ref{eq14}) and the fact that  $\sigma$ is parallel, we get from (\ref{eq20})
\begin{equation}\label{eq30} \widetilde\nabla_X (\Delta\widetilde x ) = 2c (n + 4) X + f AX - 2 c \sum_q \la X, U_q \ra U_q - f \s (X, \xi ) - 2 \s (AX, \xi ), 
\end{equation}
\noindent
where $X_{\mathcal D^{\perp}}:= \sum_q \la X, U_q \ra U_q $ is the component of the tangent vector $X$ which belongs to $\mathcal D^\perp .$
Likewise, from (\ref{eq29}) we have

\begin{eqnarray}\label{eq31} \widetilde\nabla_X (\lp^2 \widetilde x ) &=& \, [2c f^2 + 4 (n^2 + 8n + 13)] X 
- \Big\{ 4c \sum_q \alpha_q - f [f_2 + c (3n + 11)]\Big\}\, AX + 4c A^2 X  \nonumber\\
&&\quad\;\; - 2c\, [2 f_2 + f^2 + 2c (n + 7)]\, X_{\mathcal D^\perp} - 4cf \sum_q J_q AS_q X - 4c \sum_q J_q A^2 S_q X \nonumber\\  
&&\quad\;\; + \Big\{ 4c \sum_q \alpha_q - f [f_2 + c (3n + 7)]\Big\}\, \s (X, \xi )- 2\, [2f_2 + f^2 + 2c (n + 7)]\, \s (AX , \xi) \\
&&\quad\;\; - 4 f \s (A^2X, \xi ) - 4 \s (A^3X, \xi ) - 2 f \sum_i \s ((\nabla_XA)e_i , e_i) -  4 \sum_i \s ((\nabla_XA)e_i , Ae_i).\nonumber
\end{eqnarray}
\begin{table}
 \begin{center}
  \begin{tabular}[b]{|| c || c | c || c | c | c ||}
 \hline
            & $P_1^k(r)$    & $P_2 (r)$       & $H_1^k(r)$      & $H_2(r)$        & $H_3$ \\ \hline
 $\mu$    &  $\cot r $    &  $\cot r $      & $\coth r $      &  $\coth r $     &  1    \\ 
 $\nu$    &  $- \tan r$   &  $- \tan r$     & $\tanh r $      & $\tanh r$       & ---    \\
 $\alpha_1$ & $2 \cot (2r)$ & $2 \cot (2r) $  &  $2 \coth (2r)$ & $2 \coth (2r)$  &  2  \\
 $\alpha_2$ &  ---          & $- 2 \tan (2r)$ &  ---            &  $2 \tanh (2r)$ &  ---   \\ \hline
 $m(\mu)$ & $4(m-k-1)$ &  $2(m-1) $ &  $4 (m-k-1)$ & $2(m-1)$  &  $4(m-1)$ \\
 $m(\nu)$ &  $4k$  &  $2(m-1)$  &  $4k$  & $2(m-1)$ & ---  \\
 $m(\alpha_1)$ &  3  & 1  & 3  &  1  &  3  \\
 $m(\alpha_2)$ &  ---  &  2  &  ---   &  2  &  ---  \\ \hline
   \end{tabular}
  \end{center}
 \caption{Principal Curvatures of Model Hypersurfaces and Their Multiplicities }
\end{table}

Differentiating (\ref{eq26}) with respect to $X$ we get
\begin{equation}\label{eq32} \widetilde\nabla_X (\Delta^2 \widetilde x) - a \, \widetilde\nabla_X (\lp \widetilde x) + b X = 0.
\end{equation}
Therefore, using the above expressions and separating the part of (\ref{eq32}) which is tangent to $\mathbb HQ^m$ we get
\begin{eqnarray} \label{eq33} && [ 2c f^2 + 4 (n^2 + 8n + 13)  - 2c (n + 4) a + b\,] \, X + 4c\, A^2X \nonumber \\
&& - \Big\{ 4c \sum_q \alpha_q - f [f_2 + c (3n + 11) - a]\Big \}\, AX - 2c\, [2f_2 + f^2 + 2c (n + 7) - a ]\, X_{\mathcal D^\perp }\nonumber\\ 
&& - 4cf \sum_q S_q A S_q X  - 4c \sum_q S_q A^2 S_qX = 0 ,     
\end{eqnarray}
whereas the part normal to $\mathbb HQ^m$ yields
\begin{eqnarray} \label{eq34} && \Big\{ 4c\sum_q \alpha_q  - f [f_2  + c (3n + 7) - a]\Big\}\,\, \s (X, \xi) - 4 \s (A^3X, \xi )\nonumber\\
&& - 4f \s (A^2X, \xi ) - 2\, [2f_2 + f^2 + 2c (n + 7) - a] \, \s (AX, \xi ) \\
&& - 2 f \sum_i \s ((\nabla_XA) e_i , e_i) - 4 \sum_i \s ((\nabla_XA)e_i , Ae_i) = 0. \nonumber
\end{eqnarray}

Note that $J_q A^k S_q X $ has no $\xi -$component since $U_q$ is principal and $S_qU_q = 0.$ 
 
These expressions are linear in $X.$ Further separation of parts relative to the splitting
$\mathcal D\oplus \mathcal D^\perp $ of the tangent space yields the following

\begin{lemma}\label{Lem1} Let $M^n$ be a connected, curvature-adapted real hypersurface of 
$\mathbb HQ^m(4c)$ $ \, (m \geq 2, \; n = 2m-1).$ Suppose $M$ has constant principal curvatures and $\alpha_q^2 + 4 c \neq 0,$ for all $q =1, 2, 3, $\, $\alpha_q \in \mathfrak s (\mathcal D^\perp)$. If $M$ is of 2-type via $\widetilde x$ 
satisfying a 2-type condition (\ref{eq26}) then the following relations hold:
\begin{enumerate}
\item [\it{($C_1$)}] \,
\begin{eqnarray*} [2c (n + 3) + \alpha_k f]\, a &=&  b + 4 (n + 1)(n + 6) + \, \alpha_k\,\Big\{ f\, [f_2 + c (3n + 7)] - 4 c \sum_q \alpha_q \Big\} \\
                                                  &&\;\;\; - 4c f_2 + 4c \Big[ f \, \sum_q \alpha_q + \sum_q \alpha_q^2 \Big], \end{eqnarray*}
for every $\alpha_k \in \mathfrak s (\mathcal D^\perp), \,  k = 1, 2, 3;$

\item [\it{($C_2$)}] \, \begin{eqnarray*} [2c (n + 4) + \tau f]\, a &=& \;\; b + 4 (n^2 + 8n + 13)  + 4c\, \Big\{ f\Big( \tau + \sum_q \tau_q \Big) +  \Big( \tau^2 +  \sum_q {\tau_q}^2 \Big)\Big\} \\
                                                     && \;\;\; + 2cf^2 + \Big\{ f \, [f_2 + c (3n + 7)] - 4c \sum_q\alpha_q \Big\} \tau ,
\end{eqnarray*}
for any principal curvature $\tau \in \mathfrak s (\mathcal D) ;$

\item [\it{($C_3$)}] \, \begin{eqnarray*} (f + 2 \tau )\, a &=& f [f_2 + c (3n + 7)] - 4c\sum_q \alpha_q + 4 \tau\, \Big[f \Big( \tau +   \sum_q \tau _q \Big) + \Big( \tau^2 + \sum_q {\tau _q}^2\Big )\Big] \\
  && \;\;\; + \, 2 \tau \, \Big[ 2 f_2 + f^2 + 2c (n + 7) - 2f \sum_q\alpha_q - 2 \sum_q \alpha_q^2\Big] ,
\end{eqnarray*}
for any $\tau \in \mathfrak s (\mathcal D) ;$

\item [\it{($C_4$)}] \, \begin{eqnarray*} && f \la (\nabla_XA)Y , Z \ra + f \sum_q \la (\nabla_XA)(S_qY) , S_qZ \ra \\
&&\quad + \; \la (\nabla_XA^2)Y , Z \ra + \sum_q \la (\nabla_XA^2)(S_qY) , S_qZ \ra = 0, \end{eqnarray*}
for every $Y, Z \in \Gamma (\mathcal D)$ and $X \in \Gamma(TM)$.
\end{enumerate}
Conversely, if $(C_1) - (C_4)$ hold for a curvature-adapted hypersurface of $\mathbb HQ^m$ with constant principal curvatures, 
where $a$ and $b$ are constants and $\tau \in \mathfrak s (\mathcal D)$ is an arbitrary principal curvature on $\mathcal D,$ then the 
formula (\ref{eq26}) holds and the submanifold is of finite type  $\leq 2$ if the corresponding monic polynomial 
$P(t) = t^2 - a t + b $ has two distinct real roots.
\end{lemma}
The condition $\alpha_q^2 + 4 c \neq 0$ for all $q = 1, 2, 3$ is one case considered by Berndt in his analysis, under which for every $\tau \in \mathfrak s_p(\mathcal D)$ at every 
point $p$ and for every $q = 1, 2, 3,$ there exists exactly one $\tau_q \in \mathfrak s_p(\mathcal D)$ such that $S_q V_{\tau} = V_{\tau_q}$
and $\tau_q$ satisfies 
\begin{equation} \label{eq35} [2 \tau - \alpha_q]\, [2 \tau_q - \alpha_q] = \alpha_q^2 + 4c,
\end{equation}
 according to Lemma 4.9  of \cite{Ber}. Those are the $\tau_q$'s appearing in conditions $(C_2)$ and $(C_3)$. The complementary case $\alpha_q^2 + 4 c = 0$ holding for some $q = 1, 2, 3$ leads to $\alpha_1 = \alpha_2 = \alpha_3 =: \alpha$. Since it is either proved or assumed in \cite{Ber} that the principal curvatures are constant, we find from Table 1 that the only hypersurface with constant principal curvatures for which $\alpha_q^2 + 4 c = 0$ holds is the horosphere $H_3$, which is not of 2-type as will be proved in Lemma 2 below and this case can be excluded from further discussion.

\begin{proof}
Assume that the cannonical triple of quaternionc sections $J_1, J_2, J_3$ has been chosen on an open neighborhood $G_p$ of an arbitrary point $p$ of the dense set $\mathcal W \subset M$, so that $U_q, \, q = 1, 2, 3$, are principal directions. This will be the standing assumption in all subsequent calculations. By linearity of the above formulas in $X$ it suffices to consider various components of (\ref{eq33}) and (\ref{eq34}) when $X$ belongs to
all possible eigenspaces of $A.$ Note first that by (\ref{eq13}) we have
\begin{equation}\label{eq36}
\sum_q S_q A S_q U_k = \Big(\alpha_k - \sum_q \alpha_q \Big) U_k, \quad \sum_q S_q A^2 S_q U_k = \Big(\alpha_k^2 - \sum_q \alpha_q^2 \Big) U_k, \;\;  k = 1, 2, 3.
\end{equation}
Then ($C_1$) follows from (\ref{eq33}) when we set $X = U_k$ for $k = 1, 2, 3$, and use (\ref{eq36}). Next, consider $X \in V_\tau$ to be any principal direction of $A|_{\mathcal D}$ for a principal curvature 
$\tau \in \mathfrak s (\mathcal D).$ 
Since in this case $A S_q X  = \tau_q \, S_q X$ we get
\begin{equation}\label{eq37}
\sum_q S_q A S_q X = - \Big(\sum_q \tau_q\Big) X, \quad \sum_q S_q A^2 S_q X = - \Big(\sum_q \tau_q ^2\Big) X.
\end{equation}
Taking these into consideration, choose such an $X$ in (\ref{eq33}) to get ($C_2$).
Regarding the components that come from the part normal to $\mathbb HQ^m$, we recal that the normal space of $\mathbb HQ^m$ in 
$\bold H^{(1)}(m+1) = E^N_{(K)}$ is spanned by $\widetilde x $ and the values of $\s $ on various pairs of tangent vectors of $\mathbb HQ^m,$ namely by
$\s (\xi , \xi), \,  $ $\, \s (\xi , X),$ $\s (\xi , W),$ \, $\s (X, Y), \, \s (X, W), $ and $\s (V, W)$ for $X, Y \in \Gamma (\mathcal D), $ 
$V, W \in \Gamma (\mathcal D^{\perp} )$ \cite{Dim2}. Equation (\ref{eq34}) and hence equation (\ref{eq32}) have no $\widetilde x $-component since 
(\ref{eq10}) holds and $f$ and $f_2$ are constant. Note also that by (\ref{eq10})
$\s (\xi , U_q) = 0, \; \s (U_q, U_q) = \s (\xi , \xi)$ and $\s (U_q, X) = \s (\xi , J_qX)$ for $X \in \mathcal D.$ 
\par
Conditions ($C_3$) and ($C_4$) follow from the normal part (\ref{eq34}). Namely, let $\mathcal L$ denote the left-hand side of (\ref{eq34}). We consider the metric product of $\mathcal L$ with normal vectors of the above form.

By using (\ref{eq13}) - (\ref{eq15}),  the fact that $f$, $f_2$ are constant, and $U_q, q = 1, 2, 3$ are principal, it is checked in a straightforward manner that $\langle \mathcal L, \, \sigma(\xi , \xi ) \rangle = 0 $ so $\sigma (\xi , \xi)$-component
gives no additional information. By (\ref{eq10}) and symmetry of $\sigma$ we have
$$ \sigma (\xi , U_q) = \sigma (J_q\xi , J_qU_q) = - \sigma (\xi , U_q) = 0,
$$
thus by linearity $\sigma (\xi , W) = 0$ for every $W \in \Gamma (\mathcal D^{\perp })$ and there is no need to consider components of this form.

Now take $Y \in \Gamma (\mathcal D).$ We compute $\la \s (\xi , X ), \s (\xi , Y )  \ra  = c \la X, Y\ra$ and from (\ref{eq8}) and (\ref{eq15})
\begin{equation} \label{eq38}
\sum_i \la  \s (\xi , Y ), \s ((\nabla_XA)e_i, e_i) \ra  = -2c\,  \Big(\sum_q \alpha_q \Big)\la AX, Y\ra\, - 2 c\, \sum_q \la S_qAS_qAX, Y \ra ,
\end{equation}

\begin{equation}\label{eq39} 
\sum_i \langle \sigma (\xi , Y ) , \sigma ((\nabla_XA)e_i, Ae_i) \rangle  = - c\, \Big(\sum_q \alpha_q^2 \Big)\langle AX, Y\rangle\, -  c\, \sum_q \langle S_qA^2S_qAX, Y \rangle .
\end{equation}
Therefore, the inner product of (\ref{eq34}) with $\s (\xi , Y )$ gives
\begin{eqnarray*} &&4 \sum_q \la S_qA^2S_qAX, Y \ra + 4f \sum_q\la S_qAS_qAX , Y \ra - 4 \la A^3X, Y \ra - 4f \la A^2X, Y\ra\\
&& -2 \, \Big[2 f_2 + f^2 + 2c (n + 7) - a - 2f \sum_q \alpha_q - 2 \sum_q \alpha_q^2\Big] \la AX, Y\ra\\
&& + \Big\{4c \sum_q \alpha_q + f a - f\, [f_2 + c (3n + 7)] \Big\} \la X, Y\ra = 0.
\end{eqnarray*}
Since  $A\mathcal D \subset \mathcal D, \; S_q \mathcal D = \mathcal D,$ and  the expression
is linear in $X, Y \in \mathcal D$ we can drop $Y$ and take $X \in V_\tau \subset \mathcal D \,$ for any $\tau \in \mathfrak s (\mathcal D)$
using $AX = \tau X , \; A(S_qX) = \tau_q S_qX, $ with $\tau , \tau _q$ related by (\ref{eq35}), to get ($C_3$). 
\vskip 0.3truecm
Finally, to prove ($C_4$), first observe that there is no need to consider $ \s (W, Y)$-component of $\mathcal L$ for $W \in \mathcal D^{\perp}, \, Y \in \mathcal D$ since 
$W = \sum_q c_q U_q $ for some scalars $c_q$ and thus
$$ \s (W, Y) = \sum_q c_q \s (U_q, Y) = \sum_q c_q \s (\xi , J_qY),
$$
with $J_qY$ belonging also to $\Gamma (\mathcal D)$ and we alredy considered components of the form $\s (\xi , \mathcal D ).$ We also know 
that $\s (W, \xi ) = 0.$ Moreover, for $V, W \in \mathcal D^{\perp}$ we have $V = \sum_j c_j U_j, \; W = \sum_k d_k U_k, \; j, k = 1, 2, 3$ and
$$ \s (V, W) = \sum_{k=1}^3 c_k d_k \s (U_k, U_k) + \sum_{k=1}^3\sum_{j\neq k} c_j d_k \s (U_j, U_k) = \Big(\sum_{k=1}^3 c_k d_k\Big) \s (\xi , \xi),
$$
the component already considered. Therefore, there remain only components of the form $\s (Y, Z)$, \, $Y$, $Z\in \Gamma (\mathcal D)$ to be 
considered. Since $Y, Z \in \mathcal D $, by (\ref{eq8}) we have $\la \s (X, \xi ), \s (Y, Z)\ra = 0$ for any $X \in \Gamma (TM).$ Further, since $f$ and $f_2$ are constant
$$\sum_i \la \s ((\nabla_XA)e_i ,\, e_i), \s (Y, Z) \ra = 2c \,\la (\nabla_XA)Y,\, Z \ra  + 2c\, \sum_q \la (\nabla_XA)(S_qY), S_qZ \ra $$
$$\sum_i \la \s ((\nabla_XA)e_i ,\, A e_i), \s (Y, Z) \ra = c \,\la (\nabla_XA^2)Y,\, Z \ra  + c\, \sum_q \la (\nabla_XA^2)(S_qY), S_qZ \ra .$$

As the last condition, for the component in $\sigma(Y, Z)$ direction we look at  $\langle \mathcal L, \sigma (Y , Z) \rangle = 0$, where $\mathcal L$ is as above with $X \in \Gamma (TM)$ and $Y, Z \in \Gamma (\mathcal D)$ Using the preceding calculations and (\ref{eq8}),  this is easily shown to be equivalent to ($C_4$).
\par
Conversely, since we considered all possible components, the conditions ($C_1$) - ($C_4$) are equivalent to (\ref{eq33}) and (\ref{eq34}) 
by linearity and thus we get (\ref{eq32}), from which it follows that a hypersurface is of type $\leq 2,$ provided that the 
corresponding polynomial has two distinct real roots.
\end{proof}

Note that by the results of Pak \cite{Pak} for the projective case and Lyu, P\'erez and Suh \cite{LPS} for the hyperbolic case (see also \cite{AM2}), any of the class$-A$ hypersurfaces in $\mathbb HQ^m $ from either list  is
characterized by
\begin{equation}\label{eq40} (\nabla_XA)Y = - \, c\, \sum_{q = 1}^3\, [\la S_q X, Y \ra U_q + \la U_q, Y \ra S_qX], \quad X, Y \in \Gamma (TM),
\end{equation}
so that the condition ($C_4$) is trivially
satisfied for those hypersurfaces. Further, by eliminating $b$ from ($C_1$) and ($C_2$) we get
\begin{eqnarray}\label{eq41} [2c +  (\tau - \alpha_k) f ]\, a &=& \;  4(n + 7) + 2 c f^2 + 4 c f_2  + 4c \,\Big\{f \Big( \tau +  \sum_q \tau_q\Big) + \Big( \tau^2 + \sum_q \tau_q ^2 \Big)\Big\} \nonumber \\
&& + \Big\{ f\, [f_2 + c (3n + 7)] - 4 c \sum_q \alpha_q \Big\} \, ( \tau - \alpha_k) - 4 c \Big[f \sum_q \alpha_q + \sum_q \alpha_q ^2  \Big], 
\end{eqnarray} 
for any $k = 1, 2, 3 $ and $\tau \in\mathfrak s (\mathcal D)$. If $a$ can be uniquely determined from this condition (regardless of the choice of $\tau$ and $\alpha_k$ and  consistent with ($C_3$), then $b$ is
uniquely determined from ($C_1$).
\vskip 0.7truecm
\section{\bf The Main Results }
\vskip 0.6truecm
First, we show that a horosphere cannot be of 2-type and moreover

\begin{lemma}\label{lem2} 
The horosphere $H_3$ in $\mathbb HH^m$ satisfies $\Delta^2\widetilde x = \text{const} \neq 0$ and is not of any finite type via $\widetilde x$. 
\end{lemma}

\begin{proof}\,
From Table 1 we find the principal curvatures $\mu = 1, \; m(\mu) = n - 3$ and $ \alpha = 2, \; m(\alpha) = 3$ so that $f = n + 3, \; f_2 = n + 9$ and $\sum_q \alpha_q = 6$. Hence, from (\ref{eq29}) with $c = -1$ we find

\begin{align*}
\Delta^2\widetilde x \; = \; 2& (n^2 + 2n - 15)\, \xi + (n^2 + 8 n + 21)\, \sigma (\xi , \xi)\\
& + 2 (n + 4) \sum_{e_i \in \mathcal D}\sigma (e_i, e_i) - 2 (n + 3) \sum_{e_i \in \mathcal D}\sigma (e_i, e_i) - 2 \sum_{e_i \in \mathcal D}\sigma (e_i, e_i)\\
& + 2 (n + 4) \sum_q  \sigma (U_q, U_q) - 4 (n + 3) \, \sum_q  \sigma (U_q, U_q) - 8\, \sum_q  \sigma (U_q, U_q) \\
= \; \; & (n^2 + 2 n - 15)\, [\, 2 \xi +  \sigma (\xi , \xi) \,] \, ,
\end{align*}
since the terms on the second line cancel out and the terms on the third line combine with $\sigma(\xi ,\xi)$ term according to (\ref{eq10}). Hence, for $X \in \Gamma(TM)$ we have 
\begin{align*}
\widetilde\nabla_X(\Delta^2\widetilde x) &= (n + 5)(n - 3)\, [2 \sigma(X, \xi) - 2 AX - \overline A_{\sigma(\xi , \xi)} X - 2 \sigma (AX , \xi)] \\
&= \, 2 (n + 5)(n - 3)\, [\sigma (X, \xi) -  \sigma (AX, \xi) - AX + X + X_{\mathcal D^\perp }]
\end{align*}
When $X \in \Gamma(\mathcal D)$ then $AX = X$ and this expression is equal to zero,  whereas when $X \in \Gamma (\mathcal D^\perp)$ then $AX = 2 X$ so that
\begin{equation*}
\widetilde\nabla_X(\Delta^2\widetilde x) =  - 2\, (n + 5)(n - 3)\,  \sigma(X, \xi).
\end{equation*}

By (\ref{eq10}) and symmetry of $\sigma$ we have  $\sigma (U_q , \xi) = 0, \, q = 1, 2, 3, $ so that $\sigma (X, \xi) = 0 $ for $X \in \Gamma (\mathcal D^\perp)$, and hence, by linearity, valid for all tangent vectors $X$. Therefore, $\Delta^2\widetilde x =: C $ is a constant vector (matrix), i.e. $ 2 \xi + \sigma (\xi , \xi) = $ const and this constant is nonzero since $\sigma (\xi, \xi) \perp \xi $.
\par
If we assume that $H_3$ is of finite $k$-type via $\widetilde x$, then $P(\Delta)(\widetilde x - \widetilde x_0) = 0$ where $P$ is a monic polynomial of degree $k.$ Since $\Delta^s\widetilde x = 0$ for $s > 2$ we would have $C - a_1 \, \Delta\widetilde x + a_0 (\widetilde x - \widetilde x_0) = 0 $ for some constants $a_0, a_1$ where $a_1 \neq 0$ since $\widetilde x$ is not constant. Then $\widetilde\nabla_X(\Delta \widetilde x) - h X = 0$, for $X \in \Gamma (TM)$ where $h = a_0/a_1$. By (\ref{eq30}) this would imply
\begin{equation*}
- \, [2 (n + 4) +  h]\, X + 2 X_{\mathcal D^\perp} + (n + 3) AX - (n + 3)\sigma (X, \xi) - 2 \sigma (AX , \xi) = 0,
\end{equation*}
for any finite type $k \geq 1.$
Putting $X \in \mathcal D, \; AX = X$ and $X \in \mathcal D^\perp, \; AX = 2 X$ into this formula and separating parts tangent to $M$, we get $h = - (n +5)$ and $h = 0$, respectively, which is a contradiction, proving that $H_3$ cannot be of $k$-type for any finite $k$. 
\end{proof}

\begin{theorem}\label{thm1} 
There exists no real hypersurface of \,  $\mathbb HH^m(- 4), \, m \geq 2$, which is of Chen type 1 in $\bold H^1(m + 1)$ via $\Phi$. Here, we do not assume hypersurface to be curvature-adapted.
\end{theorem}

\begin{proof} \, 
Suppose that a hypersurface $M \subset \mathbb HH^m(-4)$ is of Chen-type 1 so that
$$ \widetilde x = \widetilde x_0 + \widetilde x_u, \quad \text{with}\quad \widetilde x_0 = \text{const} \quad \text{and}\quad\Delta \widetilde x_u = \lambda \widetilde x_u
$$
according to (\ref{eq1}). Then $\Delta \widetilde x - \lambda \widetilde x + \lambda \widetilde x_0 = 0$. Differentiating this relation with respect to $X \in \Gamma (TM)$ we get $\widetilde \nabla_X (\Delta \widetilde x) - \lambda X = 0$, which yields, by way of (\ref{eq20}) and (\ref{eq23}),
\begin{align*}
- (X f)\, \xi - f \sigma(X, \xi) - 2 \sigma (AX , \xi) + f AX + 2 c (n + 4) X - 2 c X_{\mathcal D^\perp} - \lambda X = 0 .
\end{align*}
Separating parts that are tangent to $M$, normal to $M$ but tangent to $\mathbb HH^m$ and normal to $\mathbb HH^m$ we get first from the $\xi$-component that $f = \text{tr} A$ is constant and 
\begin{equation}\label{eq42}
f AX + 2c (n + 4) X - 2 c X_{\mathcal D^\perp} - \lambda X = 0,
\end{equation}
\begin{equation}\label{eq43}
2 \sigma (AX , \xi) + f\, \sigma (X, \xi) = 0.
\end{equation}
Choosing $X$ in (\ref{eq42}) to belong to $\mathcal D$ and $\mathcal D^\perp$ we get, respectively,
$$ f AX + [2 c (n + 4) - \lambda] X = 0\quad \text{for}\;\; X \in \mathcal D \;\;\text{and} \quad  f AX + [2 c (n + 3) - \lambda] X = 0\quad \text{for}\;\; X \in \mathcal D^\perp. 
$$
This implies $f \neq 0$ and in our case $c = - 1$, so
$$ AX = \frac{[\lambda + 2 (n + 4)]}{f} X \quad \text{for}\;\; X \in \mathcal D \;\;\text{and} \quad   AX = \frac{[\lambda + 2 (n + 3)]}{f} X\quad\text{for}\;\; X \in \mathcal D^\perp. 
$$
Therefore, there are two constant principal curvatures
$$ \frac 1 f \, [\lambda + 2 (n + 4)]  \quad \text{and} \quad  \frac 1 f \, [\lambda + 2 (n + 3)]
$$
of respective multiplicities $n - 3$ and 3. Since $A\mathcal D \subset \mathcal D, $\; $A\mathcal D^\perp \subset \mathcal D^\perp $, $M$ is curvature adapted. Since the horosphere $H_3$ is not of 1-type by Lemma \ref{lem2}, $M$ could be geodesic hypersphere with principal curvatures $\mu = \coth r, \; \alpha = 2 \coth (2r)$ or a tube about $\mathbb HH^{m-1}$ with $\nu = \tanh r, \; \alpha = 2 \coth (2r)$. In either case $\alpha = \mu + \frac 1 \mu$, \; $f = n \mu + \frac 3 \mu$. Taking the metric product of (\ref{eq43})  with $\sigma (Y, \xi)$  with $ X, Y \in \mathcal D$ we get $AX = (- f/2) X$, i. e. $f + 2 \mu = 0$. However,
$$ f + 2 \mu = \Big( n \mu + \frac 3 \mu \Big) + 2 \mu = (n + 2) \mu + \frac 3 \mu \neq 0,
$$
since $\mu $ is positive, contradicting the above.
\par

\end{proof}

For hypersurfaces of class $A_1$ (geodesic hyperspheres and tubes about quaternionic hyperplanes) we have

\begin{lemma}\label{lem3} \begin{enumerate}
\item [\it{(i)}] A geodesic hypersphere in $\mathbb HP^m$ of any radius $r \in (0, \pi/2), \; r \neq \cot^{-1}\sqrt{3/ (4m + 1)}$ is of 2-type 
in $\bold H(m+1).$ A geodesic hypersphere in $\mathbb HH^m$ of arbitrary radius $r > 0$ is of 2-type in $\bold H^1(m + 1)$ via $\tx $ and the same 
holds for a tube of arbitrary radius $r > 0 $ about a totally geodesic quaternionic  hyperbolic hyperplane 
$\mathbb HH^{m-1} \subset \mathbb HH^m.$ These statements are also valid for any open portion of the respective submanifolds.
\item [\it{(ii)}] \, The only  mass-symmetric 2-type hypersurfaces of class $A_1$ in $\mathbb HQ^m(4c)$
are open portions of  geodesic hyperspheres of radius $r = \cot^{-1}\sqrt{1/m}$ in $\mathbb HP^m(4).$
\end{enumerate}
\end{lemma}

\begin{proof} \, 
For a geodesic sphere $P_1^0(r) $ in $\mathbb HP^m$ or $H_1^0(r)$ in $\mathbb HH^m$ (class $A_1$ in $\mathbb HP^m$ and $A^{\prime}_1$ in $\mathbb HH^m$) 
define
$$ \cot_c(r) = \begin{cases} \cot r, \; \; \;\text{when} \; c = 1 \; \text{(projective case)}\\
          \coth r, \; \text{when} \; c = - 1 \; \text{(hyperbolic case)}\end{cases}
$$
to be the circular or hyperbolic version of cotangent and denote $\tan_c(r) = 1/\cot_c(r)$.
Let $\mu = \cot_c (r)$  be the principal curvature of multiplicity $4(m-1) = n-3$ and $\alpha = 2 \cot_c (2r)$ the principal curvature of $A |_{\mathcal D^\perp}$ of multiplicity 3, whereas $\mu = \tanh r, \; \alpha = 2 \coth (2r)$ for a tube about a complex hyperbolic hyperplane 
$\mathbb HH^{m-1}(-4)$ of class $A^{\prime\prime}_1.$ Then
$$  \alpha = \mu - \frac{c}{\mu}, \quad f = n \mu - \frac{3c}{\mu}, \quad f_2 = n \mu^2 + \frac{3}{\mu^2} - 6c, \quad \mu _q = \mu , \; \forall q = 1, 2, 3 .
$$

Setting $\tau = \mu$ and $\alpha_k = \alpha$ in  (\ref{eq41}) we get
$$ [(n + 2)c - 3\mu^{-2}]\, a = c (3n+2) (n + 2) \mu^2 + (3n^2 + 2n + 4) - \frac{3c (2n+3)}{\mu^2} - \frac{9}{\mu^4}.
$$
We may assume that $(n+2)c \neq 3/\mu^2 ,$ certainly true when $c = -1,$ and when $c = 1$ the equality would lead to
$\mu = \sqrt{3/(n+2)}$, i.e. to $r = \cot^{-1}\sqrt{3/(4m+1)}.$ However, the geodesic hypersphere of this radius in $\mathbb HP^m(4)$
is of 1-type (see \cite{Dim2}). Thus, dividing by $(n+2)c - 3\mu^{-2}$ we get
\begin{equation}\label{eq44}
a = (3n + 2) \mu^2 + \frac{3}{\mu^2} + c (3n + 5) = (\mu^2 + c) \Big(3n + 2 + \frac{3c}{\mu^2}\Big).
\end{equation}
Then from ($C_1$) or ($C_2$)we find
\begin{equation}\label{eq45} 
b = 2(n + 1)\left[ n \mu^4 + c (2n+3)\mu^2 + \frac{3c}{\mu^2} + (n + 6)\right] .
\end{equation}
Setting $\tau = \mu, \, \alpha_k = \alpha$ in ($C_3$) and solving ($C_3$) for $a$  gives the same value as in (\ref{eq44}), so the conditions ($C_1$) - ($C_3$) are consistent and satisfied by the above
values of $a$ and $b,$ the condition ($C_4$) being trivially satisfied for any class-$A$ hypersurface by (\ref{eq40}). According to Lemma \ref{Lem1}, the equation (\ref{eq26}) then. Moreover, the polynomial $P(\lambda ) = \lambda^2 - a \lambda + b $ has two distinct real roots 
$\lambda_u = 2 (n + 1)(\mu^2 + c)$ and $\lambda_v = n \mu^2 + \frac{3}{\mu^2} + c (n + 3) = \frac{1}{\mu^2}(n \mu^2 + 3c)(\mu^2 + c),$ 
which are the two eigenvalues of the Laplacian from the 2-type decomposition of $\tx .$ These values are different for if 
$\lambda_u = \lambda_v$ then either $\mu^2 = 1\, $ ($c = -1$ case) or $\mu^2 = \frac{3}{n+2}$ ($c = 1$ case). First condition would lead to 
$\coth r = \pm 1$ which is not possible and the second condition gives a 1-type geodesic hypersphere. By Lemma 1 and a result of \cite{Dim2} these are not of 1-type as long as $r \neq \sqrt{3/(n + 2)}$ in the projective case and therefore all other geodesic hyperspheres and tubes about $\mathbb HH^{m - 1}$ are indeed of 2-type. We can find the actual 2-type spectral decomposition
$\widetilde x = \widetilde x_0 + \widetilde x_u + \widetilde x_v$ from
\begin{equation}\label{eq46} 
\widetilde x_u = \frac{1}{\lambda_u (\lambda_u - \lambda_v)}\,(\Delta^2 \widetilde x - \lambda_v \Delta \widetilde x) , \qquad 
\widetilde x_v = \frac{1}{\lambda_v (\lambda_v - \lambda_u)}\, (\Delta^2 \widetilde x - \lambda_u \Delta \widetilde x).
\end{equation}
Moreover, in this case (\ref{eq20}) gives
\begin{equation}\label{eq47}
\Delta \widetilde x = - \Big(n \mu - \frac{3c}{\mu}\Big)\, \xi - 3 \s (\xi , \xi) - \sum_{e_i\in \mathcal D} \s (e_i, e_i)
\end{equation}
and (\ref{eq29}) reduces to
\begin{eqnarray}\label{eq48}
\Delta^2 \widetilde x &=& - \, \Big[n^2 \mu^3 + c (3n^2 - 2n - 12)\mu - \frac{3(2n-3)}{\mu} - \frac{9c}{\mu^3}\Big]\, \xi \\
&& \;\; + \,\Big[(n^2 - 4n - 6) \mu^2 - \frac{9}{\mu^2} - 6cn\Big]\, \s (\xi , \xi) 
- 2 (n+1)(\mu^2 + c ) \sum_{e_i \in \mathcal D} \s (e_i, e_i).\nonumber
\end{eqnarray}
Then from the above formulas we get
\begin{eqnarray*} \widetilde x_u &=& \frac{1}{8m(\mu^2 + c)^2}\, \Big[- 8c(m-1) \mu\, \xi + 4(m-1) \mu^2 \s (\xi , \xi) 
- (\mu^2 + c) \sum_{e_i\in \mathcal D} \s (e_i, e_i)\Big],\\ 
 \widetilde x_v &=& - \frac{\mu}{(\mu^2 + c)^2} \, [(\mu^2 - c)\, \xi + \mu \s (\xi, \xi )].
\end{eqnarray*}

Using the expression
\begin{equation}\label{eq49}  \widetilde x = \frac{1}{m+1} I - \frac{c}{8(m+1)} \Big[  4\, \s (\xi , \xi ) + \sum_{e_i\in \mathcal D} \sigma (e_i, e_i)\Big]
\end{equation}
obtained from Lemma 2 of \cite{Dim2},  
we find the center of mass $\widetilde x_0 = \tx - \widetilde x_u - \widetilde x_v$ to be
$$ \widetilde x_0 = \frac{I}{m+1} + \frac{m \mu^2 - c}{m (\mu^2 + c)^2}\, \Big[\mu\, \xi + \frac 1 2 \s (\xi, \xi) + (\mu^2 + c)\Big(\tx - \frac{I}{m+1}\Big) \Big].$$
Since in this equation $\xi$-component is the only part tangent to $\mathbb HQ^m,$ it follows that the submanifold is mass symmetric
($\widetilde x_0 = I/(m+1)$) if and only if $m \mu^2 = c.$ So, $c = 1$ and mass-symmetric 2-type geodesic hyperspheres exist in $\mathbb HP^m(4)$ only and are 
of radius $r = \cot^{-1}\Big( \frac{1}{\sqrt m}\Big).$ 
\vskip 0.3truecm
\end{proof}

\vskip 0.2truecm
Of the two eigenvalues $\lambda_u,  \, \lambda_v$ of the Laplacian for a geodesic sphere given above, in the quaternionic hyperbolic space $\lambda_v$ is the smaller one, whereas in the projective space which one of the two is smaller and which one is bigger depends on whether the radius of a geodesic sphere is smaller or greter than $\cot^{-1}\sqrt{3/(n + 1)}.$ In any case, one can obtained upper estimates for the first two eigenvalues $\lambda_1, \, \lambda_2$ of the Laplacian on a geodesic sphere in $\mathbb HQ^m(4 c )$. In the projective case we have $\lambda_1 \leq \min \{\lambda_u,\; \lambda_v \}$ and $\lambda_2 \leq \max \{\lambda_u,\; \lambda_v \}$. In the hyperbolic case, $\lambda_1 \leq \Big(n - \frac 3 {\mu^2}\Big)(\mu^2 - 1) $ and $\lambda_2 \leq 2 (n + 1)(\mu^2 - 1)$ where $\mu = \coth r$. I the case of a tube about a canonically embedde $\mathbb HH^{m-1} \subset \mathbb HH^m$, $\lambda_u, \lambda_v$ are negative and $\lambda_u < \lambda_v$

\vskip 0.3truecm\noindent
\begin{lemma}\label{Lemma 4} $(i)\; $ There are no 2-type hypersurfaces in $\mathbb HH^m(-4)$ of class $A_2,$ i.e. no 2-type tubes about canonically 
embedded $\mathbb HH^k \subset \mathbb HH^m, \; 1 \leq k \leq m-2.$ A hypersurface of class $A_2$ in $\mathbb HP^m(4)$ is of 2-type if and only 
if it is an open portion of either (a) the tube of radius $r = \cot^{-1}\sqrt{\frac {k+1}{m-k}}$ or (b) the tube of radius 
$r = \cot^{-1}\sqrt{\frac{4k+3}{4(m-k)+ 1}},$ about a canonically embedded, totally geodesic  $\, \mathbb HP^k(4) \subset \mathbb HP^m(4),$ 
for any $k = 1, 2, ..., m-2.$
\vskip 0.3truecm\noindent
$(ii)\; $ The only mass-symmetric 2-type hypersurfaces of class $A_2$ are those in the first series of tubes (a) above.
\end{lemma}

\begin{proof} $(i)\; $Let $\mu = \cot r, \; \nu = - \tan r = - \frac{1}{\mu}$ for model hypersurface of class 
$A_2$ in $\mathbb HP^m$
and $\mu = \coth r, \; \nu = \tanh r =  \frac{1}{\mu}$ for model hypersurface of class $A_2$ in $\mathbb HH^m.$
Then $\mu, \; \nu$
have respective multiplicities $4l$ and $4k$ for some positive integers $k, \, l$ with $l = m - k - 1$ i.e. $n = 4l + 4k + 3.$ Moreover, from Table 1 and (\ref{eq35}) we get

\begin{equation}\label{eq50} 
\quad \mu\, \nu = - c, \quad \alpha = \mu - \frac{c}{\mu}= \mu + \nu , \quad \mu_q = \mu, \quad \nu_q = \nu, \;\ \forall q = 1, 2, 3.
\end{equation}
\begin{equation}\label{eq51}
f = L \mu + K \nu, \quad f^2 = L^2\mu^2 + K^2 \nu^2 - 2c KL, \quad f_2 = L \mu^2 + K \nu^2 - 6 c, 
\end{equation}
where $K:= 4k + 3$ and $L:= 4l + 3.$ Our goal is to examine when the three equations ($C_1$) - ($C_3$) are consistent and when constants $a$ and $b$
can be found to satisfy them (As we know, the condition ($C_4$) is satisfied by every class$-A_2$ hypersurface ). That comes down to the pair of
equations consisting of ($C_3$) and (\ref{eq41}), having the same solution for $a$ for either value of $\tau \in \{\mu, \nu \}.$
Consider the equation (\ref{eq41}) in which $\tau = \mu,$
multiplied by $[2c + f (\nu - \alpha )] = (2c - f \mu) $ and the same equation with $\tau = \nu$ multiplied by $(2c - f \nu).$
Subtract the two multiplied equations to eliminate $a.$ We get
\begin{equation}\label{eq52} 
f \, [f_2 + f^2 - c (n + 1) ] + 2 \alpha f \, (f + \alpha) - \, 4\, c\, \alpha = 0. 
\end{equation}
This is a necessary and sufficient condition for $a$ to have the same value from (\ref{eq41}), regardless of the choice of $\tau \in \mathfrak s(\mathcal D).$ 
On the other hand, subtracting the two equations obtained from (\ref{eq41}) for $\tau = \mu, \, \nu,$ gives
\begin{equation}\label{eq53}
a f = f \, [f_2 + c (3n + 23) ] + \, 4\, c\, \alpha. 
\end{equation}
Similarly, from the two equations contained in ($C_3$) for $\tau = \mu, \; \nu$ by subtracting we get
\begin{equation}\label{eq54}
a = 2 f_2 + f^2 + 2 \alpha (f + \alpha) + 2c (n + 11), 
\end{equation}
and by eliminating $a$ from these two equations we get exactly the same condition  (\ref{eq52}) as before. Moreover, assuming (\ref{eq52}), 
we check that (\ref{eq53}) and (\ref{eq54}) are consistent, so there is only one condition, namely (\ref{eq52}), to be satisfied in order 
to make ($C_1$) - ($C_3$) consistent,
regardless of the choice of $\tau ,$ and enable us to solve for $a$ and $b$. Substituting  the values from (\ref{eq50}) and (\ref{eq51}) 
into (\ref{eq52}), using
$\alpha = \mu + \nu$ we get
\begin{eqnarray*}
 0\;  = &\, L (L + 1)(L+2)\mu^3 + K (K+1)(K+2)\nu^3\\
& \, - c\, \mu\, (3 L^2K + 3 L^2 + 6LK + 8L + 2K + 4 )\\
& \, - c\, \nu\, (3 LK^2 + 3 K^2 + 6LK + 8K + 2L + 4).
\end{eqnarray*}
With $\nu = - c/\mu$ this yields
\begin{equation}\label{eq55}
 [(L+1)\mu^2 - c (K+1)]\,[ L(L+2)\,\mu^4 - 2c (LK + K + L + 2)\,\mu^2 + K(K+2)] = 0,
\end{equation}
which has the following three solutions
$$ (a) \; \mu^2 = \frac{(K+1)c}{L + 1} \qquad\quad (b) \; \mu^2 = \frac{Kc}{L + 2} \qquad\quad
(c) \; \mu^2 = \frac{(K+2)c}{L}. \qquad\quad
$$
Clearly, when $c = - 1$ none of them is possible, so there are no 2-type hypersurfaces of $\mathbb HH^m(-4)$ among class$-A_2$ hypersurfaces.
When $c = 1$ from (\ref{eq54}) we find
\begin{equation}\label{eq56}
a = (L^2 + 4L + 2)\,\mu^2 + (K^2 + 4K + 2)\, \nu^2 - 2 LK.
\end{equation}
Since $L + K = n + 3 = 4m + 2$ and with $c = 1$ , $2 c (n + 3) + \alpha f = (L + K) + L \mu^2 + K \nu^2$, substituting these and $a$ from (\ref{eq56}) into ($C1$) we obtain
\begin{align*}
b = & \;\;   L(L + 1)(L + 2) \mu^4 + (L^3 - L^2 K + 2 L^2 + 2 L K + 2 L + 2 K)\, \mu^2\\
 &+  K (K + 1)(K + 2) \nu^4 + (K^3 - L K^2 + 2 K^2 + 2 L K + 2 L + 2 K) \, \nu^2 \\
& - L^2 K - L K^2 - L^2 - K^2 + 4 L K + 2 L + 2 K.
\end{align*}
\par
From these and $a = \lambda_u + \lambda_v , \quad b = \lambda_u \lambda_v$, we find the two eigenvalues of the Laplacian 
from the 2-type decomposition to be
\begin{eqnarray}\label{eq57} 
\lambda_u &=& (L+1)(L+2)\mu^2 + (K+1)(K+2)\nu^2 - (L + K + 2LK), \\
\lambda_v &=& L\mu^2 + K\nu^2 + L + K, \qquad \; \mu = \cot r, \;\; \nu = - \tan r . \nonumber
\end{eqnarray}
In the case $(a)$, we get $\lambda_u = 2(n + 5) = 8 (m + 1), \; \lambda_v = 2 (n+3) - \frac{(l-k)^2}{(l+1)(k+1)} = \frac{8(m+1)(LK+ 2m + 1)}{(L+1)(K+1)},$ 
$\; \lambda_u > \lambda_v, $ so the
hypersurface is of 2-type. Since $\mu^2 = \cot^2r = \frac{K+1}{L+1},$ it follows from the Bernd's list that the hypersurface is congruent to an 
open portion of the tube of radius
$r = \cot^{-1}\sqrt{\frac{K+1}{L+1}} = \cot^{-1}\sqrt{\frac{k+1}{m-k}}$ about a canonically embedded
$\mathbb HP^k(4) \subset \mathbb HP^m(4),$ for any  $ k = 1,..., m-2\,$. In case $(b)$, (\ref{eq57}) yields
$$ \lambda_u = 8 (m+1)\frac{K+1}{K} = \frac{8 (n + 5) (k + 1)}{4 k+3} , \quad  \lambda_v = 8 (m+1) \frac{L+1}{L+2} = \frac{8 (n + 5) (l + 1)}{4 l + 5} , 
$$
\noindent
$\lambda_u > \lambda_v.$
Since $\mu^2 = \cot^2r = \frac{K}{L+2},$ we identify such hypersurface as an open portion of the tube of radius
$ r = \cot^{-1}\sqrt{\frac{K}{L+2}} = \cot^{-1}\sqrt{\frac{4k+3}{4(m-k)+1}}$ about a canonically embedded
$\mathbb HP^k(4) \subset \mathbb HP^m(4),$ for any  $ k = 1,..., m-2.$ In $\mathbb HP^m(4)$ cases (b) and (c) generate the same set of examples each being a tube of appropriate radius over one of a pair of alternate focal submanifolds $\mathbb HP^k$ and $\mathbb HP^l$. Thus, there is no need to consider the last case. 
\vskip 0.5truecm\noindent
$(ii)\;$ For a class$-A_2$ hypersurface we have from (\ref{eq57}) 
\begin{equation}\label{eq58} 
\lambda_u - \lambda_v = (L^2 + 2L + 2)\, \mu^2 + (K^2 + 2K + 2)\, \nu^2 -  2 (L+K+  LK).
\end{equation}
Note that from Table 1  and the accompanying discussion, in addition to principal curvature $\alpha = 2 \cot_c(2r)$ 
an $A_2$-hypersurface has also two more principal curvatures $\mu = \cot_c r$ and $\nu = - c\tan_c r,$ 
with corresponding principal subspaces $V_\mu$ and $V_\nu, $ being $J_q$-invariant.
Then from (\ref{eq20}) and (\ref{eq29}) for a split basis of principal directions $\{ e_i\}$ in $\mathcal D = = V_\mu \oplus V_\nu$ we get
$$ \Delta \widetilde x = - (L \mu + K \nu) \xi -3  \sigma (\xi , \xi) - \sum_{e_i \in V_\mu} \s (e_i, e_i) - \sum_{e_j \in V_\nu} \s (e_j, e_j), 
$$ 
\begin{align*}
\Delta^2 \widetilde x = & -\, \Big\{L^2 \mu^3 + K^2 \nu^3 + [ L^2 + 4(2m - 1) L - 12]\, \mu + [ K^2 + 4(2 m - 1) K - 12]\, \nu\Big\}\; \xi\\
&+ \Big[(L^2- 4 L -  6) \mu^2 + (K^2 -  4 K -  6)\nu^2 - 2LK \Big]\, \s (\xi , \xi)\\
&- 2 (L+1)(\mu^2 + 1) \sum_{e_i \in V_\mu} \sigma (e_i, e_i) - 2 (K+1)(\nu^2 + 1)\sum_{e_j \in V_\nu} \s (e_j, e_j).
\end{align*}
Then $\widetilde x_u $ and $\widetilde x_v $ can be computed as in (\ref{eq46}). Since we now assume the hypersurface of $\mathbb HP^m$ to be of type 2 and mass-symmetric 
via $\widetilde x$ we must have $\widetilde x_0 = \widetilde x - (\widetilde x_u + \widetilde x_v) = I/(m+1).$ Because $I$ and $\widetilde x$ are normal to $\widetilde x (\mathbb HP^m),$ a necessary condition 
for mass-symmetry in $\bold H^{(1)}(m+1)$ is that the $\xi$-component of $\widetilde x_u + \widetilde x_v$ be zero. 
\vskip 0.2truecm
Observing the corresponding values of $\lambda_u , \, \lambda_v$ in each of the cases we see that the $\xi$-component of $\widetilde x_u + \widetilde x_v$
for hypersurfaces in $(b)$ is never zero, whereas for hypersurfaces of case $(a)$ this component is identically equal to zero. 
Namely, since $a = \lambda_u + \lambda_v$ and $b = \lambda_u \lambda_v$ from (\ref{eq46}) we have
\begin{equation}\label{eq59}
\widetilde x_u + \widetilde x_v = \frac 1{\lambda_u \lambda_v} \Big[ (\lambda_u + \lambda_v) \Delta \widetilde x - \Delta^2 \widetilde x \Big] = \frac 1 b\, (a\, \Delta\widetilde x - \Delta^2 \widetilde x).
\end{equation}
From this and the preceding formulas, the $\xi$-component of $- (\widetilde x_u + \widetilde x_v)$ of any hypersurface belonging to case (b) when multiplied by $\nu$ is found to be equal to $64 (m + 1)/b K \neq 0$, whereas using the values $\mu^2 = \frac{K + 1}{L + 1}, \; \nu^2 = \frac{L + 1}{K + 1}$  to reduce powers $\mu^3$ and $\nu^3$ to multiples of $\mu$ and $\nu$, $K + L = 4m + 2$, and (\ref{eq20}), (\ref{eq29}), (\ref{eq51}), refering to (\ref{eq56}) we compute the $\xi$-component of (\ref{eq59}) for any hypersurface in (a)  to be 
$$ \frac 2 b \Big\{(L + 1)[LK + 4 (m - 1)] \mu  + (K + 1) [LK + 4 (m - 1)] \nu\Big\} ,
$$which yields 0 when multiplied through by $\mu$.
\par
By (\ref{eq49}) we know that
\begin{equation}\label{eq60}
\frac I{m + 1} - \widetilde x = \frac 1 {8(m + 1)} \Big[4 \sigma (\xi , \xi) + \sum_{e_i \in V_\mu}\sigma (e_i, e_i) + \sum_{e_j \in V_\nu}\sigma (e_j, e_j)  \Big].
\end{equation}
Then, further straightforward computations show that $\sigma (\xi , \xi)$-component and $\sum \sigma(e_i, e_i)$ components of $- (\widetilde x_u + \widetilde x_v)$ where  $e_i$ belongs to $V_{\mu}$ or $V_{\nu}$  are equal to corresponding components of the right-hand side of (\ref{eq60}), so all the components
on both sides of a mass-symmetric 2-type decomposition of $\widetilde x$ are matched and with this, part $(ii)$ is also proved. 
\end{proof}
\vskip 0.3truecm
The two families of tubes referred to in Lemma 4 have also another representation.
Let
$$M_{4k+3, 4l+ 3}(r):= S^{4k+3}(\cos r) \times S^{4 l+3}(\sin r), \; 0 < r < \pi / 2, $$
be the family of generalized Clifford tori in an
odd-dimensional sphere $S^{n+4} \subset \mathbb H^{m+1}, \, n = 4m-1.$ By choosing the two spheres (with the indicated radii) in the above
product to lie in quaternionic subspaces we get the fibration $S^3 \to M_{4k+3, 4 l+3}(r) \to M^{\mathbb H}_{k, l}(r):= \pi (M_{4k+3, 4 l+3}(r))$ 
compatible with the Hopf fibration $\pi : S^{n+4} \to \mathbb HP^m(4),$ which submerses $M_{4k+3, 4 l+3}(r)$ onto $M^{\mathbb H}_{k, l}(r)$.
P\'erez and Santos show \cite{PS} that $M^{\mathbb H}_{k, l}(r)$ is a tube of radius $r$ about totally geodesic $\mathbb HP^k(4)$
with principal curvatures $\cot r, \; - \tan r,\; 2 \cot (2r)$ of respective multiplicities $4 l, 4 k,$ and $3.$ Accordingly, the family
of hypersurfaces corresponding to the case $(a)$ is given as open portions of
\vskip 0.2truecm
$   \qquad\qquad   M^{\Bbb H}_{k, l}(r) = \pi \left(S^K\left( \sqrt{\frac{K+1}{n+3}}\right)\times
S^L\left( \sqrt{\frac{L+1}{n+3}}\right)  \right),\quad \cot^2r = \frac{K+1}{L+1},$
\newline\noindent
and the family of hypersurfaces corresponding to the case $(b)$ is
\vskip 0.2truecm
$ \qquad\qquad  M^{\Bbb H}_{k, l}(r) = \pi \left(S^K\left( \sqrt{\frac{K}{n+3}}\right)\times
S^L\left( \sqrt{\frac{L+2}{n+3}}\right)  \right), \quad \cot^2r = \frac{K}{L+2},$
\newline\noindent
where for both families $K = 4 k+ 3$ and $L = 4 l+ 3$ are odd positive integers with $K + L = n + 3 = 4 m + 2.$ The family of hypersurfaces corresponding to the case $(c)$ is the same family as 
in $(b),$ with the roles of $K$ and $L$ interchanged and the spherical factors in the above representation reversed. Hypersurfaces of case $(c)$ can be also described as tubes over 
$\mathbb HP^k(4)$ of radius $ r = \cot^{-1}\sqrt{\frac{4k+5}{4(m-k)-1}},$ for $k = 1, 2, \dots , m-2, $ but are not examined as a separate
case since they constitute the same family as the one under case $(b)$. Namely, the tube about $\mathbb HP^k(4)$ of this radius $r$ is the 
same as the tube over the other focal variety $\mathbb HP^l(4)$ of radius $\frac{\pi}{2} - r = \cot^{-1}\sqrt{\frac{4 l+3}{4(m-l)+1}},$ 
which appears within family $(b).$
\vskip 0.4truecm\noindent
{\bf Remark.} Note that according to a result of Barbosa et al. \cite{BdCE}, the tube of radius $r$ over $\mathbb HP^k(4)$  in $\mathbb HP^m(4)$ \, is stable with respect to normal
variations preserving the enclosed volume if and only if 
$ \cot^{-1}\sqrt{\frac{4 k+5}{4(m-k)-1}} \leq \, r\, \leq \cot^{-1}\sqrt{\frac{4 k+ 3}{4(m-k)+1}} \;$  . Hence, the 2-type tubes over $\mathbb HP^k(4)$ of radii $ \cot^{-1}\sqrt{\frac{4 k+ 3}{4(m-k)+1}}$
and $ \cot^{-1}\sqrt{\frac{4k+5}{4(m-k)-1}}$
are distinguished by being the largest, respectively the smallest, stable tubes about $\mathbb HP^k, $ for each  $ k = 1, 2, ..., m-2,$
i.e. the values of radii in cases $(b)$ and $(c)$ are precisely the endpoints of the stability interval for $r.$
\vskip 0.4truecm
\begin{lemma}\label{Lemma 5}
There exist no class-B hypersurfaces of $\mathbb H H^m(- 4)$ which are of 2-type via $\widetilde x$, i.e. no 2-type hypersurface $H_2(r)$ in $\mathbb HH^m(- 4)$ as given in Table 1. A class-B hypersurface of $\mathbb H P^m(4)$ is of 2-type via $\widetilde x$ if and only if it is an open portion of one of the following
\begin{enumerate}
\item [\it{(i)}] \, The tube of radius $r = \frac 1 2 \cot^{-1}(1/\sqrt m) \, $  about a canonically embedded, totally geodesic  $\, \mathbb CP^m(4) \subset \mathbb HP^m(4);$ 

\item [\it{(ii)}] \, The tube of radius $r = \frac 1 2\cot^{-1}\sqrt{(3 + \sqrt{96 m^2 - 15})\Big/2 (4m^2 - 1)}\, $ about a canonically embedded, totally geodesic  $\, \mathbb CP^m(4) \subset \mathbb HP^m(4).$ 
\end{enumerate}
\vskip 0.2truecm
In both of these cases, the hypersurface is also mass-symmetric.
\end{lemma}

\begin{proof}
For class-$B$ hypersurfaces and the values of their principal curvatures given in Table 1, observing the notation introduced in Lemma 3 we have the principal curvatures of $A|_{\mathcal D}$
$$ \mu = \cot_c (r), \; \nu = - c \tan_c (r), \quad \text{both of multiplicity} \quad 2 (m - 1)
$$
and those of $A|_{\mathcal D^\perp}$,
$$ \alpha := \alpha_1 = 2 \cot_c( 2 r), \quad \alpha_2 = \alpha_3 = - 2 c \tan_c (2 r),
$$
where $\alpha$ has multiplicity 1 and the other one multiplicity 2. Then
\begin{equation}\label{eq61}
\mu \nu = - c, \quad \mu + \nu = \alpha, \quad \alpha_2 = \alpha_3 = - 4 c/\alpha
\end{equation}
\begin{equation}\label{eq62}
\sum_{q = 1}^3 \alpha_q = \alpha - \frac{8 c}\alpha , \qquad \sum_{q = 1}^3 \alpha_q^2 = \alpha^2 + \frac{32}{\alpha^2}
\end{equation}
\begin{equation}\label{eq63}
f = (2 m - 1) \alpha - \frac{8c} \alpha , \qquad f^2 = (2m - 1)^2 \alpha^2 + \frac{64}{\alpha^2} -  16 c (2 m - 1)\end{equation}
\begin{equation}\label{eq64}
f_2 = (2m - 1) \alpha^2 + \frac{32}{\alpha^2} + 4 c (m - 1) . \hskip 4truecm
\end{equation}
For every $\tau \in \mathfrak s (\mathcal D)$ we find from (\ref{eq35}) that $\tau_q = (2 c + \alpha_q \tau)/(2 \tau - \alpha_q)$, from which we get
\begin{equation}\label{eq65} \mu_1 = \mu ,\;\; \mu_2 = \mu_3 = \nu \quad \text{and} \quad \nu_1 = \nu, \;\; \nu_2 = \nu_3 = \mu
\end{equation}
and hence for both $\tau = \mu$ and $\tau = \nu$ we have
\begin{equation}\label{eq666}
\tau + \sum_q \tau_q = 2 \alpha, \qquad \tau^2 + \sum_q \tau_q ^2 = 2 \alpha^2 + 4 c \, .
\end{equation}

In $(C_1)$ substitute $\alpha_1$ and  $\alpha_2$ for $\alpha_k$ and subtract the two formulas to get
\begin{align}\label{eq67}
f a &= f\, [f_2 + c (3n + 7)] - 4 c\sum_q \alpha_q \\
&= (2m - 1)^2 \alpha^3 + 4 c (8 m^2 - 8 m + 1) \alpha - \frac{64 m}{\alpha} - \frac{256 c}{\alpha^3}. \nonumber
\end{align}
and then substututing this back into $(C_1)$ we get
\begin{equation}\label{eq68}
(n + 3) a = \frac b 2 \, c + 2 \alpha^2 + \frac{128}{\alpha^2} + 8 c (4 m^2 + 1).
\end{equation}
In $(C_2)$ cancell out $\tau fa = \tau \Big\{f\, [f_2 + c (3 n + 7)] - 4 c \sum_q \alpha_q \Big\}$ on both sides, use (\ref{eq62}) - (\ref{eq66}) and $n = 4m - 1$ to get
\begin{equation}\label{eq69}
(n + 4) a = \frac b 2\, c + (4 m^2 + 4 m + 1) \alpha^2 + \frac{64}{\alpha^2} + 4 c (8 m^2 + 4 m + 1)
\end{equation}
Subtracting (\ref{eq68}) from (\ref{eq69}) gives
\begin{equation}\label{eq70}
a = (4 m^2 + 4 m - 1) \alpha^2 - \frac{64}{\alpha^2} + 4 c (4 m - 1)
\end{equation}
and substituting this in either (\ref{eq68}) or (\ref{eq69}) yields
\begin{equation}\label{eq71}
b = 8 c (m + 1) \Big[(4 m^2 + 2 m - 1) \alpha^2 - \frac{64}{\alpha^2} + 4 c (2m - 1)  \Big]
\end{equation}
Joint validity of the formulas (\ref{eq67}) - (\ref{eq71}) makes $(C_1)$ and $(C_2)$ consistent. Next, cancel out $f a$ from both sides of $(C_3)$ using (\ref{eq67}) and use (\ref{eq62}) - (\ref{eq66}) on the right hand side to simplify and solve for $a$ to get exactly the same value as in (\ref{eq70}), which means that $(C_3)$ is consistent with the previous formulas. Finally, to make (\ref{eq67}) consistent with (\ref{eq70}), multiply the latter by $f$ from (\ref{eq63}) to obtain 
\begin{equation}\label{eq72}
f a = (2m - 1)(4 m^2 + 4m - 1)\, \alpha^3 - 4 c (14m - 3)\, \alpha - \frac{32(8m- 3)}{\alpha} + \frac{512 c}{\alpha^3} ,
\end{equation}
which must be the same as $(\ref{eq67})$ for suitable values of $\alpha^2$. Equating $fa$ from these two equations results in the equation
\begin{equation}\label{eq73} m (4 m^2 - 1)\, \alpha^6 - 4 c (4m^2 + 3m - 1)\, \alpha^4 - 48 (2m - 1)\, \alpha^2 + 384 c = 0 .
\end{equation}
The left-hand side factors as
$$ (m \alpha^2 - 4c)[(4m^2 - 1) \alpha^4 - 12 c \alpha^2 - 96 ],
$$
from where we get the solutions of (\ref{eq73}) to be
$$ \alpha^2 = \frac{4 c}m, \quad \alpha^2 = \frac{6 c \pm 2 \sqrt{96 m^2 - 15}}{4m^2 - 1}.
$$
Since in the hyperbolic case $c = - 1$ and $\alpha^2 = 4 \coth^2(2 r) > 4$, we see that none of these works in that case, so there are no hypersurfaces of class $B$ in $\mathbb HH^m(- 4)$ which are of 2-type. In the projective case  ($c = 1$) we get two solutions for $\alpha = 2\cot{2r}$ and the corresponding values for the radii
$$ r = \frac 1 2\,  \text{arccot}\left(\frac 1{\sqrt m}\right)\, , \;\; r = \frac 1 2\, \text{arccot}\left(\sqrt{\frac{3 + \sqrt{96 m^2 - 15}}{2(4 m^2 - 1)}}\;\right) ,
$$ 
which both satisfy the conditions $(C_1)$ - $(C_3)$ of Lemma 1. There is also the condition $(C_4)$ which is not automatically satisfied since the relation (\ref{eq40}) does not hold for class-$B$ hypersurfaces. However, we will show  by direct calculation that the tubes of these two radii about canonically embeded $\mathbb CP^m(4) \subset \mathbb HH^m(4) $ are mass-symmetric and of 2-type by verifying that the mass-symmetric 2-type equation 
\begin{equation}\label{eq74} \Delta^2 \widetilde x - a \Delta \widetilde x + b \left(\widetilde x - \frac{I}{m+1}  \right) = 0
\end{equation}
holds for these hypersurfaces. Namely, assume that the basis of quaternionic structure $\{ J_q \}, q = 1, 2, 3$ has been chosen locally so that $U_q = - J_q \xi$ is a principal direction for $\alpha_q$ and $\{e_i  \}$ is a basis of principal directions. Using (\ref{eq62}) - (\ref{eq64}), formula (\ref{eq20}) gives
\begin{equation}\label{eq75}
\Delta \widetilde x = - \Big[(2m - 1)\alpha - \frac {8c}{\alpha}  \Big]\, \xi - 3 \sigma (\xi, \xi) - \sum_{e_i \in \mathcal D} \sigma (e_i, e_i)
\end{equation}
and (\ref{eq29}) reduces to
\begin{align}\label{eq76}
\Delta^2 \widetilde x = &- f a \, \xi + \Big\{(4 m^2 - 4 m - 1) \alpha^2 - \frac{64}{\alpha^2} - 4 c (4m + 1)  \Big\}\, \sigma (\xi, \xi) \nonumber\\
& - 2 c (4m + 3) \sum_{e_i \in \mathcal D} \sigma (e_i, e_i ) - 2 f\,\sum_{e_i \in \mathcal D} \sigma (e_i, Ae_i ) -  2 \sum_{e_i \in \mathcal D}\sigma (Ae_i, Ae_i ) ,
\end{align}
where $fa$ is given by (\ref{eq67}). Let the notation $\{ \cdot \}_{mod\,  \Xi}$ mean taking only terms of the expression $\{ \cdot \}$ that involve $\xi$ and $\sigma (\xi, \xi)$ and $\{ \cdot \}_{mod\,  {\Xi}^\perp}$
taking those terms that involve $\sigma (e_i, e_i)$.
Since the multiplicity of both $\mu$ and $\nu$ is $2(m - 1)$, from the discussion about the formula (\ref{eq35}) and using (\ref{eq3}) and (\ref{eq65}) it follows that there is an orthonormal basis of principal directions of $\mathcal D$,
\begin{equation*}
\{e_i \} \cup \{J_1 e_i \} \cup \{J_2 e_i  \} \cup \{J_3 e_i \}, \quad i = 1, 2, \dots , m-1
\end{equation*}
so that $\{e_i \} \cup \{J_1 e_i \}$ spans $V_{\mu}$ and $\{J_2 e_i  \} \cup \{J_3 e_i \}$ spans $V_{\nu}, \; i = 1, 2, \dots , m-1$.
Then, using this basis for $\mathcal D$ and observing (\ref{eq61}), (\ref{eq63}) we compute
\begin{align*}
\{\Delta^2 \widetilde x  \}_{mod\,  {\Xi}^\perp} &=  -  2 \,  \Big\{ (4m + 3) c\sum_{\substack{{\tau = \mu , \nu}\\ e_i \in V_\tau}}  \sigma(e_i, e_i) + \sum_{\substack{{\tau = \mu , \nu}\\ e_i \in V_\tau}}  (f \tau + \tau^2)\,  \sigma (e_i, e_i) \Big\}\\
&= -  2 \,  \Big\{ (4m + 3) c \sum_{i = 1}^{m - 1}  4 \, \sigma(e_i, e_i) +  (f \mu + \mu^2)\sum_{i = 1}^{m - 1}  \,  2\,  \sigma (e_i, e_i) +  (f \nu + \nu^2)\sum_{i = 1}^{m - 1}  \,  2\,  \sigma (e_i, e_i)    \Big\} \\
&= - 2\, \Big \{4c (4m + 3) + 2 \, f (\mu + \nu) + 2 (\mu^2 + \nu^2)  \Big \} \sum_{i = 1}^{m - 1} \sigma (e_i, e_i)\\
&= - 8m (\alpha^2 + 4c) \, \sum_{i = 1}^{m - 1} \sigma (e_i, e_i) = - 2m (\alpha^2 + 4c) \, \sum_{e_j \in \mathcal D} \sigma (e_j, e_j) \, ,
\end{align*}
and, likewise, from (\ref{eq20}) and (\ref{eq49})
$$ \{\Delta \widetilde x  \}_{mod\,  {\Xi}^\perp} = -  \, \sum_{e_i \in \mathcal D } \sigma (e_i, e_i), \qquad  \Big\{ \widetilde x - \frac{I}{m + 1} \Big\}_{mod\,  {\Xi}^\perp} = - \frac c{8 (m + 1)} \sum_{e_i \in \mathcal D} \sigma (e_i, e_i).
$$
Therefore,  when $c = 1$, with values of $a$ and $b$ from (\ref{eq70}) and (\ref{eq71}) it follows that $mod\,  {\Xi}^{\perp}-$ component of the left-hand side of (\ref{eq74}) is zero and in a similar way, using (\ref{eq67}), (\ref{eq76}), \, $mod\,  {\Xi} \, -$ component is also zero, which means that equation (\ref{eq74}) is satisfied, showing that the immersion $\widetilde x$ is mass-symmetric and of 2-type. For the hypersurface $M$ given in $(i)$ we have 
$$ \lambda_u + \lambda_v = a = \frac{4 (m + 1)(4m - 1)}m , \qquad \lambda_u \lambda_v = b = \frac{32 (m + 1)^2(2m - 1)}m,
$$
hence $\lambda_u = 4 (m + 1)(2m - 1)/ m$ and $\lambda_v = 8 (m + 1)$ belong to the spectrum of $M$. Moreover, using
$$ \widetilde x_u = \frac 1{\lambda_u - \lambda_v}\Big[ \Delta \widetilde x - \lambda_v \Big(\widetilde x - \frac{I}{m + 1}  \Big) \Big], \qquad \widetilde x_v = \frac 1{\lambda_v - \lambda_u}\Big[ \Delta \widetilde x - \lambda_u \Big(\widetilde x - \frac{I}{m + 1}  \Big) \Big]
$$
and (\ref{eq49}) we find an explicit spectral resolution of immersion $\widetilde x$ into vector eigenfunctions:
\begin{align*}
\widetilde x_u &= -\,  \frac{\sqrt m}{2(m + 1)}\, \xi - \frac{m}{4 (m + 1)} \, \sigma (\xi , \xi)\\
\widetilde x_v &= \frac{\sqrt m}{2 (m + 1)} \, \xi + \frac{m - 2}{4 (m + 1)}\, \sigma (\xi , \xi) - \frac 1 {8(m + 1)} \sum_{e_i \in \mathcal D} \sigma (e_i, e_i).
\end{align*}
plus the center of mass $\widetilde x_0 = I/(m + 1).$
\end{proof}
One example of a  mass-symmetric 2-type hypersurface of class $B$ is obtained when we take $m = 3$ to get an 11-dimensional real hypersurface of $\mathbb HP^3$ which is the tube of radius $r = \pi/6$ about a canonically embedded $\mathbb CP^3$. When $m = 4$, the second value of $r$ gives an example of 15-dimensional hypersurface of 2-type in $\mathbb HP^4$, which is the tube of radius $\pi/6$ about $\mathbb CP^4$.
\par
Except for the two tubes given in Lemma 5, all other hypersurfaces of class $B$ in $\mathbb HP^m$ are of 3-type.
\vskip 0.3truecm
\begin{lemma}\label{Lemma 6}
A tube of any radius $r > 0$ about canonically embedded $\mathbb CH^m(- 4)$ in $\mathbb HH^m(4)$ is mass-symmetric and of 3-type. The same is true for a tube of any radius $r \in (0, \pi/4)$ about canonically embedded $\mathbb CP^m(4)$ in $\mathbb HP^m(4)$, except for the two tubes listed in Lemma 5 which are of 2-type.
\end{lemma}

\begin{proof}
The proof goes by direct computation of iterated Laplacians and matching corresponding componnets of 3-type mass-symmetric equation. We begin by computing the Laplacian of $\sigma(\xi, \xi)$. Since
$$ \widetilde\nabla_X \sigma(\xi, \xi) = - 2 c\,  [X + \sum_{q = 1}^3 \langle U_q, X\rangle U_q ] - \, 2\, \sigma (AX, \xi),
$$
by repeated use (\ref{eq11}) and (\ref{eq12}) we further compute
\begin{align*}
\Delta (\sigma (\xi, \xi)) = &\sum_{i = 1}^n \Big[ \widetilde\nabla_{\nabla_{e_i}e_i} \sigma (\xi , \xi) - \widetilde\nabla_{e_i}\widetilde\nabla_{e_i} \sigma (\xi , \xi) \Big ]  \\
= & \; 2 c \Big[ f  + \sum_q \alpha_q \Big]\,  \xi + 2 c \sum_{q = 1}^3 \Big[ \nabla_{U_q} U_q + \sum_{i = 1}^n \langle  \nabla_{e_i} U_q, \, e_i\rangle\, U_q \Big]\\
& - \, 2 \,\sum_{i = 1}^n \, \overline A_{\sigma (A e_i , \xi) e_i} + 2 (3 c + f_2)\, \sigma (\xi , \xi)\\
&\;  +\, 2 \, \sum_{i = 1}^n \sigma ((\nabla_{e_i}A)e_i , \, \xi) + 2 c \, \sum_{i = 1}^n \sigma (e_i, e_i) - 2 \sum_{i = 1}^n \sigma (A e_i, A e_i).
\end{align*}
Then employing (\ref{eq15}) and (\ref{eq16}) we find
$$  \sum_{q = 1}^3 \Big[ \nabla_{U_q} U_q + (\text{div}\, U_q)\, U_q \Big] = 0, \qquad \sum_{i = 1}^n (\nabla_{e_i}A) e_i = \text{tr} (\nabla A) = \nabla f = 0
$$
and using (\ref{eq9}) and (\ref{eq62}) - (\ref{eq64}) arrive at
\begin{equation}\label{eq77}
\Delta (\sigma (\xi , \xi)) = 4c \, \Big(\alpha - \frac{8 c}{\alpha} \Big)\, \xi + 4\, [(2m + 1) c + (m - 1)\alpha^2]\, \sigma (\xi , \xi) - \alpha^2 \, \sum_{e_i \in \mathcal D} \sigma (e_i, e_i).
\end{equation}
For a curvature-adapted hypersurface we have $\sum_q  S_q A U_q = 0$. Using $(\ref{eq61}) - (\ref{eq65}) $ combined with (\ref{eq77}), formula (\ref{eq24}) yields
\begin{align}\label{eq78}
\sum_{e_i \in \mathcal D} \Delta (\sigma (e_i , e_i)) = & \; 8 c (m - 1)\Big[(2m + 3) \alpha - \frac{8 c}{\alpha}  \Big] \, \xi - 8 (m - 1)(c + 2 \alpha^2) \, \sigma (\xi , \xi) \nonumber\\
& \;\; + 4\, \Big[ 2 c (m + 1) + \alpha^2 \Big] \, \sum _{e_i \in \mathcal D} \sigma (e_i,  e_i) .
\end{align}
Further, the expressions (\ref{eq20}), (\ref{eq22}) and (\ref{eq29}) reduce, respectively, to (\ref{eq75}) and

\begin{equation}\label{eq79}
\Delta \xi =   \; \Big[ (2m - 1) \alpha^2 + \frac{32}{\alpha^2} + 8 c (m -1)  \Big] \, \xi  - \Big[(2m - 3) \alpha + \frac{8 c}{\alpha}  \Big] \, \sigma (\xi , \xi) + \alpha \, \sum_{e_i \in \mathcal D} \sigma (e_i,  e_i), 
\end{equation}

\begin{align}\label{eq80}
\Delta^2 \widetilde x = & \; - \Big[ (2m - 1)^2 \alpha^3 + 4 c (8 m^2 - 8 m + 1)\, \alpha - \frac{64 m}{\alpha} - \frac{256 c}{\alpha^3}  \Big] \, \xi \nonumber\\
& \; + \Big[(4 m^2 - 4 m - 1) \alpha^2 - \frac{64}{\alpha^2} - 4 c (4 m + 1)  \Big] \, \sigma (\xi , \xi) \\
& \; - 2 m (\alpha^2 + 4 c) \sum_{e_i \in \mathcal D} \sigma (e_i,   e_i). \nonumber
\end{align}
Taking the Laplacian of (\ref{eq80}) and using (\ref{eq77}) - (\ref{eq79}) we have

\begin{align}\label{eq81}
\Delta^3 \widetilde x = & \; - \Big\{(2m - 1)^3 \alpha^5 + 8 c\, (16 m^3 - 20 m^2 + 6m  - 1)\, \alpha^3    \nonumber\\
& \qquad\quad + 16 \, (24 m^3 - 28 m^2 + 6 m  - 1)\, \alpha - \frac{512 c (2m - 1)}{\alpha} - \frac{4096 m}{\alpha^3} - \frac{8192 c}{\alpha^5} \Big\} \; \xi \nonumber\\
& \;  + \Big\{ (24 m^3 - 20 m^2 - 6m + 1)\, \alpha^4 + 8 c \, (12 m^3 - 8 m^2 - 6 m + 1)\, \alpha^2 \\
&\qquad\quad - \frac{512 c (3m - 1)}{\alpha^2} - \frac{2048}{\alpha^4} + 16\, (4 m^2 - 30 m + 17)\Big\} \; \sigma (\xi , \xi) \nonumber\\
& \; - \Big\{ 8 m^2 \alpha^4 + 48 c m^2 \alpha^2 - \frac{256 c}{\alpha^2} + 64 (m^2 - 1) \Big\}\, \sum_{e_i \in \mathcal D} \sigma (e_i , e_i). \nonumber
\end{align}
It is now a direct verification that $\widetilde x$ satisfies the equation for a mass-symmetric 3-type hypersurface, viz.
$$ \Delta^3 \widetilde x + p \Delta^2 \widetilde x + q \Delta \widetilde x + r \Big(\widetilde x - \frac{I}{m + 1}  \Big) = 0
$$
where 
$$p = - \frac 1{\alpha^2} (\alpha^2 + 4 c) [(6 m - 1) \alpha^2 + 8 c] , \;  \; r = - \frac{16 c}{\alpha^2}\, (2 m^2 + m - 1)(\alpha^2 + 4 c)^2 (m \alpha^2 + 4 c)
$$.
$$  q = \frac 1{\alpha^2} (\alpha^2 + 4 c) [4m (2m - 1) \alpha^4 + 8 c (6 m^2 + 2 m - 1) \alpha^2 + 32 (4m + 1)].
$$
Given that there are no 1-type hypersurfaces of this kind, when we eclude those two 2-type examples from Lemma 5, we are left with all the other such hypersurfaces which must be of type 3.
\end{proof}
\par
Combining together the information from Lemmas 2-5 we obtain
\begin{theorem}\label{thm2}
Let $M^{4m - 1}$ be a curvature adapted real hypersurface of $\mathbb HP^m(4), \ m \geq 2.$ Then  $M$ is of 2-type in $\bold H (m + 1)$ via $\Phi$ if and only if $M$ is locally congruent to an open portion of one of the following
\begin{enumerate}  
\item [\it{(i)}]\,  A geodesic hypersphere of $\mathbb HP^m(4)$ of any radius $r \in (0, \pi/2), \; r \neq \cot^{-1}\sqrt{3/ (4m + 1)}$;

\item [\it{(ii)}] \,  The tube of radius $r = \cot^{-1}\sqrt{\frac {k+1}{m-k}}$ about a canonically embedded, totally geodesic  $\, \mathbb HP^k(4) \subset \mathbb HP^m(4),$ 
for any $k = 1, 2, ..., m-2$;

\item [\it{(iii)}] \, The tube of radius $r = \cot^{-1}\sqrt{\frac{4k+3}{4(m-k)+ 1}} \,$ about a canonically embedded, totally geodesic  $\, \mathbb HP^k(4) \subset \mathbb HP^m(4),$ 
for any $k = 1, 2, ..., m-2;$

\item [\it{(iv)}] \, The tube of radius $r = \frac 1 2 \cot^{-1}(1/\sqrt m) \, $  about a canonically embedded, totally geodesic  $\, \mathbb CP^m(4) \subset \mathbb HP^m(4);$ 

\item [\it{(v)}] \, The tube of radius $r = \frac 1 2\cot^{-1}\sqrt{(3 + \sqrt{96 m^2 - 15})\Big/ 2 (4m^2 - 1)}\, $ about a canonically embedded, totally geodesic  $\, \mathbb CP^m(4) \subset \mathbb HP^m(4) $. 

\end{enumerate}
\end{theorem}
\begin{proof}
As proved in \cite{Ber}, a curvature-adapted hypersurface of $\mathbb HP^m(4)$ has constant principal curvatures and is to be found in the left portion of Table 1. Then, the Lemmas 3-5 finish the proof.
\end{proof}
Likewise,

\begin{theorem}\label{thm3}
Let $M^{4m - 1}$ be a curvature adapted real hypersurface of $\mathbb HH^m(- 4)$ with constant principal curvatures. Then  $M$ is of 2-type in $\bold H^1 (m + 1)$ via $\Phi$ if and only if $M$ is locally congruent to an open portion of one of the following
\begin{enumerate}  
\item [\it{(i)}]\,  A geodesic hypersphere of $\mathbb HH^m (- 4)$ of any radius $r > 0$;

\item [\it{(ii)}] \,  The tube of arbitrary radius $r > 0 $ about a totally geodesic quaternionic  hyperbolic hyperplane $\mathbb HH^{m-1}(- 4) $ of $ \mathbb HH^m(- 4)$. 
\end{enumerate}

\end{theorem}

\begin{theorem}\label{thm4}
Among 2-type hypersurfaces of $\mathbb HQ^m(4 c)$ listed in Theorems 2 and 3

\begin{enumerate}  
\item [\it{(i)}]\,  The only ones that are of 2-type and mass-symmetric in an appropriate hyperquadric of $\mathbf H^{(1)} (m + 1)$ centered at $I/(m + 1)$ that contains them are open portions of: \, (1)\, geodesic hyperspheres of $\mathbb HP^m (4)$ of radius $r = \cot^{- 1}(\sqrt{1/ m})$ (2)\, the tubes of radius $r = \cot^{-1}\sqrt{\frac {k+1}{m-k}}$ about a canonically embedded, totally geodesic  $\, \mathbb HP^k(4) \subset \mathbb HP^m(4),$ 
for any $k = 1, 2, ..., m-2$ and (3)\, the tubes of radii $r = \frac 1 2 \cot^{-1}(1/\sqrt m) \, $ and $r = \frac 1 2\cot^{-1}\sqrt{(3 + \sqrt{96 m^2 - 15})\Big/ 2 (4m^2 - 1)}\, $ about a canonically embedded, totally geodesic  $\, \mathbb CP^m(4) \subset \mathbb HP^m(4).$ 

\item [\it{(ii)}] \,  The 2-type hypersurfaces that are minimal in $\mathbb HQ^m(4c)$ exist only in the quaternionic projective space and are open portions of geodesic spheres of radius $r = \cot^{-1}(\sqrt{3/(4 m - 1)})$ and tubes of radius $r = \pi/4$ around canonically  embedded $\mathbb HP^k (4)$ in $\mathbb HP^{2 k + 1}(4)$ when the quaternionic dimension $m = 2k + 1$ is odd.
\end{enumerate}

\end{theorem}
\begin{proof}
The proof of part $(i)$ is contained in Lemmas 3-5. To prove part $(ii)$ we explot the information given in Table 1. Regarding class$-A_1$ hypersurfaces when $k = 0$ (geodesic spheres) we have $\mu = \cot_c r, \, m(\mu) = 4 (m - 1)$ and $\alpha = 2 \cot_c (2r), \, m(\alpha) = 3$ as principal curvatures, hence 
$$f = \text{tr}\, A = 4 (m - 1) \cot_c r + 6 \cot_c(2r) = (4 m - 1)\cot_c r - 3c \tan_c r,
$$
due to $2 \cot_c(2 r) = \cot_c r - c \tan_c r$.
So, $\text{tr}\, A$ can be zero only when $c = 1$ and $\cot^2 r= 3/(4m - 1) = 3/n$, yielding the geodesic sphere in $\mathbb HP^m(4)$ of radius $r = \cot^{-1} \sqrt{3/(4m - 1)} $ which is indeed of 2-type and minimal. If $k = m - 1$ in the projective case $P_1^{m - 1}(r) = P_1^0 (\frac{\pi}2 - r)$ is still a geodesic sphere, yielding the same value of the radius as above. In the hyperbolic space the principal curvatures are $\nu = \tanh r$ and $\alpha = 2 \coth (2 r)$ in which case $\text{tr}\, A = (4 m - 1) \tanh r + 3 \coth r$ cannot be zero, so no minimal hypersurface of this kind. Searching for minimal 2-type examples among class$-A_2$ hypersurfaces  we see that formula (\ref{eq53}) implies that $\alpha = \mu + \nu = \mu - \frac c{\mu} = 0$, i. e. $\mu^2 = c$. So when $c =  - 1$ we cannot have a minimal 2-type hypersurface in $\mathbb HH^m$ of class $A_2$. In the projective case it follows that $\mu = \cot r = 1$ and $\nu = - 1$ so $r = \pi/4$ and $f = L\mu + K \nu = L - K = 0$, implying that $l = k = \frac{m - 1}2$
so $m = 2 k + 1$ is necessarily odd, resulting in class$-A_2$ hypersurface $M^{\mathbb H}_{k, k} (\pi/4) = \pi \Big(S^{4 k + 3}(1/2) \times S^{4k + 3}(1/2) \Big)$, which is indeed minimal and of 2-type  as first observed by Chen \cite[1st ed, pp. 265-266 ]{Che}, using a different method. It is the tube of radius $\pi/4$ about a canonically embedded totally geodesic $\mathbb HP^k (4) \subset \mathbb HP^{2k + 1}(4)$. Our argument shows,  that it is in fact the only 2-type minimal hypersurface of $\mathbb HQ^m(4c)$ among complete class$-A_2$ hypersurfaces. So, minimal 2-type hypersurfaces of class $A_2$ exist only in the projective space and only when $m = 2k + 1$ is odd and they are congruent to (an open portion of) the hypersurface  $M^{\mathbb H}_{k, k}(\pi/4)$ described above, which also happens to be mass-symmetric according to Lemma 4.
\end{proof}
Curvature-adapted assumption on a hypersurface is seemingly a strong condition, at least in the $\mathbb HP^m$ case, since it implies constant principal curvatures. On the other hand, one does not need this, or any other assumption if the number of principal curvatures of a hypersurface of $\mathbb HQ^m$ is restricted to be at most two. Namely, when $m \geq 3$, Martinez and P\'erez showed in $\mathbb HP^m$ case and Ortega and P\'erez in $\mathbb HH^m$,  that a hypersurface with at most two distinct principal curvatures at each point  must be an open subset of one of the following:
 (1)\, A geodesic hypersphere (2) \, A tube over totally geodesic $\mathbb HH^{m - 1}(- 4)$
 in $\mathbb HH^m(- 4)$ (3)\, A horosphere $H_3$ in $\mathbb HH^m (- 4)$.
\par
Hence,
\vskip 0.3truecm\noindent
{\bf Corollary 1.} If $M$ is a real hypersurface of $\mathbb HQ^m(4c), \, m \geq 3$, with at most two distinct principal curvatures at each point and $M$ is not a horosphere, then $M$ is of Chen type 1 when it is (an open portion of) a geodesic hypersphere of radius
 $\cot^{-1}(3/\sqrt{4m+ 1})$ in $\mathbb HP^m(4)$ or, otherwise, of Chen type 2.

\vskip 0.2truecm
\hskip 1.4truecm {\it Ivko Dimitri\'c}
\vskip 0.1truecm
\hskip 1truecm {\sc  Pennsylvania State University Fayette}
\vskip 0.1truecm
\hskip 1truecm {\sc  2201 University Drive,}
\vskip 0.1truecm
\hskip 1truecm {\sc  Lemont Furnace, PA 15456, USA.}
\vskip 0.1truecm
\hskip 1truecm \it Email: \; {\it ivko\symbol{"40}psu.edu}

\end{document}